\documentclass[reqno]{amsart}
\usepackage{float}
\usepackage{amsmath}
\usepackage{graphicx}
\usepackage{latexsym}
\usepackage{amsfonts}
\usepackage{amssymb}
\theoremstyle{plain}
\newtheorem{theorem}{Theorem}

\newtheorem{lemma}[theorem]{Lemma}
\newtheorem{proposition}[theorem]{Proposition}
\theoremstyle{definition}

\newtheorem{example}[theorem]{Example}

\newtheorem{remark}[theorem]{Remark}
\newtheorem*{remark*}{Remark}

\begin{document}
\title[Conditional limit theorems for ordered random walks]{Conditional limit theorems for ordered random walks}
\author[Denisov]{Denis Denisov}
\address{School of MACS, Heriot-Watt University, Edinburgh EH14 4AS, UK}
\email{denisov@ma.hw.ac.uk}

\author[Wachtel]{Vitali Wachtel}
\address{Mathematical Institute, University of Munich, Theresienstrasse 39, D--80333
Munich, Germany}
\email{wachtel@mathematik.uni-muenchen.de}
\begin{abstract}
In a recent paper of Eichelsbacher and K{\"o}nig (2008) the model of ordered random walks has been considered.
There it has been shown that, under certain moment conditions, one can construct a $k$-dimensional random walk conditioned
to stay in a strict order at all times. Moreover, they have shown that the rescaled random walk converges to the 
Dyson Brownian motion. In the present paper we find the optimal moment assumptions for the construction of the
conditional random walk and generalise the limit theorem for this conditional process. 
\end{abstract}

%\version\vspace{1cm}

\maketitle

%%%%%%%%%%%%%%%%%%%%%%%%%%%%%%%%%%%%%%%%%%%%%%%%%%%%%%%%%%%%%%%%%%%%%%%%%%%%%%%%%%%%%%%%%%%%%%%%%%%%%%%%%%
\section{Introduction, main results and discussion}

\subsection{Introduction}
A number of important results have been recently proved relating the limiting distributions 
of random matrix theory with certain other models. These models include the longest increasing 
subsequence, the last passage percolation, non-colliding particles, the tandem queues, random tilings, 
growth models and many others. A thorough review of these results can be found in \cite{K05}. 

Apparently it was Dyson who first established a connection between random matrix theory and non-colliding
particle systems. It was shown in his classical paper~\cite{Dy62} that the process of eigenvalues of the Gaussian Unitary Ensemble
of size $k\times k$ coincides in distribution with the $k$-dimensional diffusion, which can be represented as the evolution
of $k$ Brownian motions conditioned never to collide. Such conditional versions of random walks have attract a lot of attention in the recent past, see e.g. \cite{OY02,KOR02}. The approach in these papers is based on explicit formulas for nearest-neighbour random walks. 
However, it turns out that the results have a more general nature, that is, they remain valid for random walks with arbitrary jumps,
see \cite{BS06} and \cite{EK08}. The main motivation for the present work was to find minimal conditions, under which one can define 
multidimensional random walks conditioned never to collide.

Consider a random walk $S_n=(S_n^{(1)},\ldots,S_n^{(k)})$ on 
$\mathbf R^k$, where 
$$
S_n^{(j)}=\xi_1^{(j)}+\cdots+\xi_n^{(j)}, \quad j=1,\ldots, k,
$$
and $\{\xi^{(j)}_n,1\leq j\leq k,\ n\geq1\}$ is a family of independent and identically distributed random variables.
Let 
$$
W=\{x=(x^{(1)},\ldots,x^{(k)})\in\mathbf R^k: x^{(1)}<\ldots<x^{(k)}\}
$$
be the Weyl chamber.

In this paper we study the asymptotic behaviour of the random walk $S_n$  conditioned to stay in $W$.
Let $\tau_x$ be the exit time from the Weyl chamber of the random walk with starting point $x\in W$, that is,
$$
\tau_x=\inf\{n\ge 1: x+S_n\notin W\}.
$$

One can attribute two different meanings to the words 'random walk conditioned to stay in $W$.' On the one hand, the
statement could refer to the path $(S_0,S_1,\ldots,S_n)$ conditioned on $\{\tau_x>n\}$. On the other hand, one can 
construct a new Markov process, which never leaves $W$. There are two different ways of defining such a conditioned processes.
First, one can determine its finite dimensional distributions via the following limit
\begin{equation}
\label{dist.lim}
\mathbf {P}_x\left(\widehat{S}_i\in D_i,\,0\leq i\leq n\right)=
\lim_{m\to\infty}\mathbf P(x+S_i\in D_i,\,0\leq i\leq n | \tau_x>m ).
\end{equation}
Second, one can use an appropriate Doob $h$-transform. 
If there exists a function $h$ (which is usually called \emph{invariant function}) such that $h(x)>0$ for all $x\in W$ and 
\begin{equation}
\label{inv}
\mathbf E[h(x+S(1));\tau_x>1]=h(x),\quad x\in W,
\end{equation}
then one can make a {\it change of measure} 
$$
\mathbf {\widehat P}_x^{(h)}(S_n\in dy)=
\mathbf P(x+S_n\in dy,\tau_x>n)\frac{h(y)}{h(x)}.
$$
As a result, one obtains a random walk  $S_n$ under a new measure $\mathbf {\widehat P}_x^{(h)}$. This transformed random walk is
a Markov chain which lives on the state space $W$. 

To realise the first approach one needs to know the asymptotic behaviour of $\mathbf{P}(\tau_x>n)$.
And for the second approach one has to find a function satisfying (\ref{inv}). It turns out that
these two problems are closely related to each other: The invariant function reflects the dependence
of $\mathbf{P}(\tau_x>n)$ on the starting point $x$. Then both approaches give the same Markov chain.
For one-dimensional random walks conditioned to stay positive it was shown by Bertoin and Doney \cite{BD94}.
They proved that if the first moment of a random walk is finite, then the function 
$V(x)=x-\mathbf{E}(x+S_{\sigma_x})$ is invariant and that $\mathbf{P}(\sigma_x>n)\sim  CV(x)\mathbf{P}(\sigma_0>n)$,
where $\sigma_x=\min\{k\geq1:x+S_k\leq 0\}$. The analogous program for random walks in the Weyl chamber was carried out
by Eichelsbacher and K{\"o}nig \cite{EK08}. If we define the direct analogue of the invariant function used by
Bertoin and Doney as follows
\begin{equation}\label{defn.V}
 V(x)= \Delta(x)-\mathbf E\Delta(x+S_{\tau_x}),
\end{equation}
where $\Delta(x)$ denotes the Vandermonde determinant, that is,
$$
\Delta(x)=\prod_{1\leq i<j\leq k}(x^{(j)}-x^{(i)}),\quad x\in W.
$$
Then it was shown in \cite{EK08} 
that if $\mathbf{E}|\xi|^{r_k}<\infty$ with some $r_k>ck^3$, then 
it can be concluded that $V$ is a finite 
and strictly positive invariant function. Moreover, the authors determined the 
behaviour of $\mathbf{P}(\tau_x>n)$ and studied some asymptotic properties
of the conditioned random walk. They also posed a question about minimal moment 
assumptions under which one can construct a conditioned random walk by using $V$. 
In the present paper we answer this question. We prove that the 
results of \cite{EK08} remain valid under the following conditions:
\begin{itemize}
\item {\it Centering assumption:} We assume that $\mathbf E\xi=0$. 
\item {\it Moment assumption:} We assume that $\mathbf E|\xi|^{\alpha}<\infty$ with $\alpha=k-1$ if $k>3$ and some $\alpha>2$ if $k=3$.
Furthermore, we shall assume, without loss of generality, that $\mathbf{E}\xi^2=1$.
\end{itemize}
It is obvious, that this moment condition is the minimal one for the finiteness of the function $V$
defined by (\ref{defn.V}). Indeed, from the definition of $\Delta$ it is not difficult to see that the 
finiteness of the $(k-1)$-th moment of $\xi$ is necessary for the finiteness of $\Delta(x+S_1)$. 
Thus, this moment condition is also necessary for the integrability of $\Delta(x+S_{\tau_x})$, which
is equivalent to the finiteness of $V$. In other words, if $\mathbf{E}|\xi|^{k-1}=\infty$, then one has 
to define the invariant function in a different way.
Moreover, we give an example, which shows that if the moment assumption does not hold,
then $\mathbf{P}(\tau_x>n)$ has a different rate of divergence.
%%%%%%%%%%%%%%%%%%%%%%%%%%%%%%%%%%%%%%%%%%%%%%%%%%%%%%%%%%%%%%%%%%%%%%%%%%%%%%%%%%%%%%%%%%%%%%%%%%%%
\subsection{On the tail of $\tau_x$}
Here is our \emph{main} result:

\begin{theorem}\label{T}
Assume that $k\geq3$ and let the centering as well as the moment assumption hold.
Then the function $V$ is finite and strictly positive. Moreover, as $n\to\infty$,
\begin{equation}\label{eq.asym.tau}
 \mathbf P(\tau_x>n)\sim\varkappa V(x) n^{-k(k-1)/4},\quad x\in W,
\end{equation}
where $\varkappa$ is an absolute constant.
\end{theorem}
All the claims in the theorem have been proved in \cite{EK08} under more restrictive assumptions: 
As we have already mentioned, the authors have assumed that $\mathbf{E}|\xi|^{r_k}<\infty$ with 
some $r_k$ such that $r_k\geq ck^3$, $c>0$. Furthermore, they needed some additional regularity 
conditions, which ensure the possibility to use an asymptotic expansion in the local central 
limit theorem. As our result shows, these regularity conditions are superfluous and one needs $k-1$ 
moments only.

Under the condition that $\xi^{(1)},\ldots, \xi^{(k)}$ are identically distributed, the centering assumption
does not restrict the generality: One has only to change to the random walk $S_n-n\mathbf{E}\xi$. But if the drifts 
are allowed to be unequal, then the asymptotic behaviour of $\tau_x$ and that of the conditioned random walk
might be different, see \cite{PR08} for the case of the Brownian motion. 

We now turn to the discussion of the moment condition in the theorem. We start with the following example.
\begin{example}
Assume that $k\geq4$ and consider the random walk, which satisfies
\begin{equation}\label{Ex}
\mathbf{P}(\xi\geq u)\sim u^{-\alpha}\quad\text{as }u\to\infty,
 \end{equation}
with some $\alpha\in(k-2,k-1)$. Then,
\begin{eqnarray*}
\mathbf P(\tau_x>n) \ge 
\mathbf P\left(\xi^{(k)}_1>n^{1/2+\varepsilon}, \min_{1\le i\le n} S^{(k)}_i>0.5n^{1/2+\varepsilon}\right)\\
\times
\mathbf P\left(\max_{1\le i\le n} S_{i}^{(k-1)}\le 0.5n^{1/2+\varepsilon},
\tilde{\tau}_{x}>n\right),
\end{eqnarray*}
where $\tilde{\tau}_x$ is the time of the first collision in the random walk $(S_n^{(1)},\ldots,S_n^{(k-1)})$.
Now, by the Central Limit Theorem,
\begin{align*}
&\mathbf P\left(\xi^{(k)}_1>n^{1/2+\varepsilon}, \min_{1\le i\le n} S^{(k)}_i>0.5n^{1/2+\varepsilon}\right)\\
&\hspace{1cm}\ge\mathbf P\left(\xi^{(k)}_1>n^{1/2+\varepsilon}\right) 
\mathbf P\left(\min_{1\le i\le n} (S^{(k)}_i-\xi_1^{(k)})>-0.5n^{1/2+\varepsilon}\right)
\sim n^{-\alpha(1/2+\varepsilon)}.
\end{align*}
The CLT because is applicable due to the condition $\alpha>k-2$, which implies the finiteness of the variance.

For the second term in the product we need to analyse $(k-1)$ random walks under the condition $\mathbf E|\xi|^{k-2+\varepsilon}<\infty$. 
Using Theorem~\ref{T}, we have
$$
\mathbf P(\tilde{\tau}_{x}>n)\sim \tilde{V}(x)n^{-(k-1)(k-2)/4}. 
$$
Since $S_n$ is of order $\sqrt{n}$ on the event $\{\tilde{\tau}_{x}>n\}$, we have
\begin{eqnarray*}
\mathbf P\left(\max_{1\le i\le n} S_{i}^{(k-1)}\le 0.5n^{1/2+\varepsilon},\tilde{\tau}_{x}>n\right)\sim
\mathbf P\left(\tilde{\tau}_{x}>n\right)\sim 
\tilde{V}(x)n^{-(k-1)(k-2)/4}. 
\end{eqnarray*}

As a result the following estimate holds true for sufficiently small 
$\varepsilon$,
\begin{eqnarray*}
\mathbf P(\tau_x>n) \ge C(x) n^{-(k-1)(k-2)/4}n^{-\alpha(1/2+\varepsilon)}.
\end{eqnarray*}
The right hand side of this inequality decreases slower than $n^{-k(k-1)/4}$ for all sufficiently small $\varepsilon$.

Moreover, using the same heuristic arguments, one can find a similar lower bound in case 
(\ref{Ex}) holds with $\alpha\in(k-j-1,k-j)$, $j\leq k-3$:
\begin{eqnarray*}
\mathbf P(\tau_x>n) \ge C(x) n^{-(k-j)(k-j-1)/4}n^{-\alpha j(1/2+\varepsilon)}.
\end{eqnarray*}

We believe that the lower bounds constructed above are quite precise, and we conjecture that
$$
\mathbf{P}(\tau_x>n)\sim U(x)n^{-(k-j)(k-j-1)/4-\alpha j/2}
$$
in case (\ref{Ex}) holds.
\hfill$\diamond$
\end{example}

It is clear that $\mathbf E|\xi|^{k-1}<\infty$ is necessary for the finiteness of $V$.
Furthermore, the example shows that this condition is almost necessary for the validity of (\ref{eq.asym.tau}):
One can not obtain the relation $\mathbf{P}(\tau_x>n)\sim C(x)n^{-k(k-1)/4}$ assuming that
$\mathbf E|\xi|^{k-1-\varepsilon}<\infty$ with some $\varepsilon>0$. 

If we have two random walks, i.e. $k=2 $, then $\tau_x$ is the 
exit time from $(0,\infty)$ of the random walk $Z_n:=(x^{(2)}-x^{(1)})+(S_n^{(2)}-S_n^{(1)})$. 
It is well known that, for symmetrically distributed random walks, $\mathbf{E}Z_{\tau_x}<\infty$ if and only if $\mathbf{E}(\xi_1^{(2)}-\xi_1^{(1)})^2<\infty$. However, the existence of $\mathbf{E}Z_{\tau_x}$
is not necessary for the relation $\mathbf{P}(\tau_x>n)\sim C(x)n^{-1/2}$, which holds 
for all symmetric random walks. This is contary to the high-dimensional case ($k\geq4$), where the 
integrability of $\Delta(x+S_{\tau_x})$ and the rate $n^{-k(k-1)/4}$ are quite close to each other.

In case we have three random walks our moment condition is not optimal. We think that the existence of
the variance is sufficient for the integrability of $\Delta(x+S_{\tau_x})$. But our approach requires more than two moments.
Furthermore, we conjecture that, as in the case $k=2$, the tail of the distribution of $\tau_x$ is of order $n^{-3/2}$ for
$\it{all}$ random walks.
%%%%%%%%%%%%%%%%%%%%%%%%%%%%%%%%%%%%%%%%%%%%%%%%%%%%%%%%%%%%%%%%%%%%%%%%%%%%%%%%%%%%%%%%%%%%%%%%%%%%%%%%%%%%%%%%%%%%%%%%%%%
\subsection{Scaling limits of conditioned random walks}
Theorem \ref{T} allows us to construct the conditioned random walk via the distributional limit (\ref{dist.lim}).
In fact, if (\ref{eq.asym.tau}) is used, we obtain, as $m\to\infty$,
\begin{align*}
\mathbf{P}(x+S_n\in D|\tau_x>m)&=\frac{1}{\mathbf{P}(\tau_x>m)}\int_D\mathbf{P}(x+S_n\in dy)\mathbf{P}(\tau_y>m-n)\\
&\to\frac{1}{V(x)}\int_D\mathbf{P}(x+S_n\in dy)V(y).
\end{align*}
But this means that the distribution of $\widehat{S}_n$ is given by the Doob transform with function $V$.
(This transformation is possible, because $V$ is well-defined, strictly positive on $W$ and satisfies
$\mathbf{E}[V(x+S_1);\tau_x>1]=V(x)$.)
In other words, both ways of construction described above give the same process.

We now turn to the asymptotic behaviour of $\widehat{S}_n$. To state our results we introduce the limit process.
For the $k$-dimensional Brownian motion with starting point $x\in W$ one can change the measure using the Vandermonde
determinant:
$$
\mathbf {\widehat P}_x^{(\Delta)}(B_t\in dy)=
\mathbf P(x+B_t\in dy )\frac{\Delta(y)}{\Delta(x)}.
$$
The corresponding process is called Dyson's Brownian motion. Furthermore, one can define 
Dyson's Brownian motion with starting point $0$ via the weak limit of $\mathbf {\widehat P}_x^{(\Delta)}$,
for details see Section 4 of O'Connell and Yor \cite{OY02}. We will denote the corresponding probability measure 
as $\mathbf {\widehat P}_0^{(\Delta)}$.

\begin{theorem}\label{T2}
If $k\geq3$ and the centering as well as the moment assumption are valid, then
\begin{equation}\label{eq.cond.dist}
\mathbf P\left(\frac{x+S_n}{\sqrt n} \in \cdot \Big| \tau_x>n\right)\to\mu\quad\text{weakly},
\end{equation}
where $\mu$ is the probability measure on $W$ with density proportional to $\Delta(y)e^{-|y|^2/2}$.\\ 
Furthermore, the process $X^n(t)=\frac{S_{[nt]}}{\sqrt n}$  under the probability measure 
$\mathbf{\widehat P}^{(V)}_{x\sqrt n},x \in W$ converges weakly to the Dyson Brownian motion under the measure
$\mathbf{\widehat P}^{(\Delta)}_{x}$. Finally, the process $X^n(t)=\frac{S_{[nt]}}{\sqrt n}$  under the probability measure 
$\mathbf{\widehat P}^{(V)}_{x},x \in W$ converges weakly to the Dyson Brownian motion under the measure
$\mathbf{\widehat P}^{(\Delta)}_{0}$.
\end{theorem}
Relation (\ref{eq.cond.dist}) and the convergence of the rescaled process with starting point $x\sqrt{n}$ were proven in
\cite{EK08} under more restrictive conditions. Convergence towards $\mathbf{\widehat P}^{(\Delta)}_{0}$ was proven
for nearest-neighbour random walks, see \cite{OY02} and \cite{Sch09}. A comprehensive treatment of the case $k=2$ can be found
in \cite{BD06}.

One can guess that the convergence towards Dyson's Brownian motion holds even if we have finite variance only. However, it is not clear
how to define an invariant function in that case.

\subsection{Description of the approach}
The proof of finiteness and positivity of the 
function $V$ is the most difficult part of the paper. To derive these properties
of $V$ we use martingale methods. It is well known that $\Delta(x+S_n)$ is a martingale.
And in the case of a nearest-neighbour random walk, or in the case of the Brownian motion,
we can define $\tau_x$ as the first time of $\Delta(x+S_n)$ being non-positive. But in
general it could happen that $\Delta(x+S_{\tau_x})>0$. In other words, the martingale
$\Delta(x+S_n)$ does not 'feel' the stopping time $\tau_x$. So the stopping 
time $T_x=\min\{k\geq1:\Delta(x+S_k)\leq0\}$ seems to be more natural 
for the martingale $\Delta(x+S_n)$. Moreover, it helps us to obtain the desired properties
of $V$. We first show that $\Delta(x+S_{T_x})$ is integrable, which yields the integrability
of $\Delta(x+S_{\tau_x})$, see Subsection \ref{sect.integrability}. Furthermore, 
it follows from the integrability of $\Delta(x+S_{T_x})$ that the function 
$V^{(T)}(x)=\lim_{n\to\infty}\mathbf{E}\{\Delta(x+S_n),T_x>n\}$ is well defined 
on the set $\{x:\Delta(x)>0\}$. To show that the function $V$ is strictly positive, 
we use the interesting observation that the sequence $V^{(T)}(x+S_n){\bf 1}\{\tau_x>n\}$ 
is a supermartingale, see Subsection 2.2. 

It is worth mentioning that the detailed 
analysis of the martingale properties of the random walk $S_n$ allows one to keep the 
minimal moment conditions for positivity and finiteness of $V$. The authors of
\cite{EK08} used the H{\"o}lder inequality at many places in their proof. 
This explains the superfluous moment condition in their paper.

To prove the asymptotic relations in our theorems we use a version of the 
Komlos-Major-Tusnady coupling proposed in \cite{M76}, see Section 3.
A similar coupling has been used in \cite{BM05} and \cite{BS06}.
In order to have a good control over the quality of the Gaussian 
approximation we need more than two moments of the random walk. 
This fact explains partially why we required the finiteness of 
$\mathbf{E}|\xi|^{2+\delta}<\infty$ in the case $k=3$.
%%%%%%%%%%%%%%%%%%%%%%%%%%%%%%%%%%%%%%%%%%%%%%%%%%%%%%%%%%%%%%%%%%%%%%%%%%%%%%%%%%%%%%%%%%%%%%%%%%%%%%%%%%
%%%%%%%%%%%%%%%%%%%%%%%%%%%%%%%%%%%%%%%%%%%%%%%%%%%%%%%%%%%%%%%%%%%%%%%%%%%%%%%%%%%%%%%%%%%%%%%%%%%%%%%%%%
%%%%%%%%%%%%%%%%%%%%%%%%%%%%%%%%%%%%%%%%%%%%%%%%%%%%%%%%%%%%%%%%%%%%%%%%%%%%%%%%%%%%%%%%%%%%%%%%%%%%%%%%%%
\section{Finiteness and positivity of $V$}\label{sect.V.is.good}
The main purpose of the present section is to prove the following statement.
\begin{proposition}\label{prop1}
The function $V$ has the following properties:
\begin{itemize}
\item[(a)] $V(x)=\lim_{n\to\infty}\mathbf{E}[\Delta(x+S_n);\tau_x>n]$; 
\item[(b)] $V$ is monotone, i.e. if $x^{(j)}-x^{(j-1)}\leq y^{(j)}-y^{(j-1)}$ for all $2\leq j\leq k$, then
$V(x)\leq V(y)$;
\item[(c)] $V(x)\leq c\Delta_1(x)$ for all $x\in W$, where $\Delta_t(x)=\prod_{1\le i<j\le k}\left(t+|x^{(j)}-x^{(i)}|\right)$;
\item[(d)] $V(x)\sim \Delta(x)$ provided that $\displaystyle \min_{2\leq j\leq k}(x^{(j)}-x^{(j-1)})\to\infty$;
\item[(e)] $V(x)>0$ for all $x\in W$.
\end{itemize}
\end{proposition}
As it was already mentioned in the introduction our approach relies on the investigation of properties 
of the stopping time $T_x$ defined by 
$$
T_x=T=\min\{k\geq1:\Delta(x+S_k)\leq 0\}.
$$
It is easy to see that $T_x\geq\tau_x$ for every $x\in W$.
%%%%%%%%%%%%%%%%%%%%%%%%%%%%%%%%%%%%%%%%%%%%%%%%%%%%%%%%%%%%%%%%%%%%%%%%%%%%%%%%%%%%%%%%%%%%%%%%%%%%%%%%%%%%%
\subsection{Integrability of $\Delta(x+S_{T_x})$}\label{sect.integrability}
We start by showing that $\mathbf{E}[\Delta(x+S_{T_x})]$ is finite under the conditions of Theorem~\ref{T}.
In this paragraph we omit the 
subscript $x$ if there is no risk of confusion. 

\begin{lemma}\label{lem-1}
The sequence $Y_n:=\Delta(x+S_n){\rm 1}\{T>n\}$ is a submartingale.
\end{lemma}
\begin{proof}
Clearly,
\begin{align*}
\mathbf{E}\left[Y_{n+1}-Y_n|\mathcal{F}_n\right]
&=\mathbf{E}\left[\left(\Delta(x+S_{n+1})-\Delta(x+S_{n})\right){\rm 1}\{T>n\}|\mathcal{F}_n\right]\\
&\hspace{2cm}-\mathbf{E}\left[\Delta(x+S_{n+1}){\rm 1}\{T=n+1\}|\mathcal{F}_n\right]\\
&={\rm 1}\{T>n\}\mathbf{E}\left[\left(\Delta(x+S_{n+1})-\Delta(x+S_{n})\right)|\mathcal{F}_n\right]\\
&\hspace{2cm}-\mathbf{E}\left[\Delta(x+S_{n+1}){\rm 1}\{T=n+1\}|\mathcal{F}_n\right].
\end{align*}
The statement of the lemma follows now from the facts that
$\Delta(x+S_n)$ is a martingale and $\Delta(x+S_{T})$ is non-positive.
\end{proof}

For any $\varepsilon>0$, define the following set 
$$
W_{n,\varepsilon}=\{x\in R^k:|x^{(j)}-x^{(i)}|>n^{1/2-\varepsilon},1\le i<j\le k\}.
$$
\begin{lemma}\label{lem0}
For any sufficiently small $\varepsilon>0$ there exists 
$\gamma>0$ such the following inequalities hold
\begin{eqnarray}\label{eq0}
\left|\mathbf E[\Delta(x+S_T);T\le n]\right|\le\frac{C}{n^{\gamma}}\Delta(x),\quad x\in W_{n,\varepsilon}\cap\{\Delta(x)>0\}
\end{eqnarray}
and 
\begin{eqnarray}\label{eq00}
\left|\mathbf E[\Delta_1(x+S_\tau);\tau\le n]\right|\le\frac{C}{n^{\gamma}}\Delta(x),\quad x\in W_{n,\varepsilon}\cap W.
\end{eqnarray}
\end{lemma}
\begin{proof}
We shall prove (\ref{eq0}) only. The proof of (\ref{eq00}) requires some minor changes, and we omit it.

For a constant $\delta >0$, which we define later, let
$$
A_n=\left\{\max_{1\le i\le n, 1\le j\le k}|\xi_i^{(j)}|\le n^{1/2-\delta}\right\}
$$
and  split the  expectation into 2 parts,
\begin{align}\label{L0.0}
\nonumber
\mathbf E [\Delta(x+S_T);\,T\le n]
&=
\mathbf E [\Delta(x+S_T);\,T\le n, A_n]+\mathbf E [\Delta(x+S_T);\,T\le n, \overline A_n]\\
&=: E_1(x)+E_2(x).
\end{align}
It follows from the definition of the stopping time $T$ that at least one of the differences
$(x^{(r)}+S^{(r)}-x^{(s)}-S^{(s)})$ changes the sign at time $T$, i.e. one of the following events
occurs
$$
B_{s,r}:=\left\{(x^{(r)}+S^{(r)}_{T-1}-x^{(s)}-S^{(s)}_{T-1})(x^{(r)}+S^{(r)}_T-x^{(s)}-S^{(s)}_T)\leq0\right\},
$$
$1\leq s<r\leq k$. Clearly,
$$
|E_1(x)|\le 
\sum_{1\le s<r\le k}
\mathbf E [|\Delta(x+S_T)|;\,T\le n, A_n, B_{s,r}].
$$
On the event $A_n\cap B_{s,r}$,
\begin{eqnarray*}
\Big|x^{(s)}-x^{(r)}+S^{(s)}_{T}-S^{(r)}_T\Big|\leq 
\Big|\xi^{(s)}_{T}-\xi^{(r)}_{T}\Big|\le 2n^{1/2-\delta}.
\end{eqnarray*}
This implies that on the event 
$A_{n}\cap B_{s,r}$,
$$
|\Delta(x+S_T)|\le 2n^{1/2-\delta}\left| \frac{\Delta(x+S_T)}{x^{(s)}-x^{(r)}+S^{(r)}_{T}-S^{(s)}_T}\right|.
$$
Put $\mathcal P=\{(i,j), 1\le i<j\le k\}$. Then,
\begin{align*}
\frac{\Delta(x+S_T)}{x^{(s)}-x^{(r)}+S^{(r)}_{T}-S^{(s)}_T}
&=
\prod_{(i,j)\in \mathcal P\backslash (s,r)}
\left(x^{(j)}-x^{(i)}+S_T^{(j)}-S_T^{(i)}\right)
\\
&=
 \sum_{\mathcal J \subset \mathcal P\backslash(s,r)}
	\prod_{\mathcal  J} \left(x^{(i_2)}-x^{(i_1)}\right) 
	\prod_{\mathcal P\backslash (\mathcal  J\cup(s,r))} \left(S^{(j_2)}_T-S^{(j_1)}_T\right).
\end{align*}
As is not difficult to see, 
$$
\prod_{\mathcal P\backslash (\mathcal  J\cup(s,r))} (S^{(j_2)}_T- S^{(j_1)}_T)
=p_{\mathcal J}(S_T)
=\sum_{i_1,i_2,\ldots,i_k}\alpha^{\mathcal J}_{i_1,i_2,\ldots,i_k}(S^{(1)}_T)^{i_1}\ldots(S^{(k)}_T)^{i_k},
$$
where the sum is taken over all $i_1,i_2,\ldots,i_k$ such that $i_1+i_2+\ldots+i_k=|\mathcal P|-|\mathcal  J|-1$.

Put $M^{(j)}_n=\max_{0\le i \le n}{|S^{(j)}_i|}$. Combining Doob's and Rosenthal's inequalities, one has
\begin{equation}
\label{DR}
\mathbf{E}\left(M^{(j)}_n\right)^p\leq C(p)\mathbf{E}\left|S^{(j)}_n\right|^p\leq
C(p)\mathbf{E}[|\xi|^p]n^{p/2}
\end{equation}
Then, 
\begin{eqnarray}\label{L0.1}
\nonumber
\mathbf E |p_{\mathcal J}(S_T)\mathbf 1_{\{T\le n\}}|
&\le & 
 \sum_{i_1,i_2,\ldots,i_k}|\alpha^{\mathcal J}_{i_1,i_2,\ldots,i_k}|\mathbf E(M^{(1)}_n)^{i_1}\ldots\mathbf E (M^{(k)}_n)^{i_k}\\
\nonumber
&\le & 
 \sum_{i_1,i_2,\ldots,i_k}|\alpha^{\mathcal J}_{i_1,i_2,\ldots,i_k}|C_{i_1}  n^{i_1/2} \ldots C_{i_k}  n^{i_k/2}\\
&\le &
C_{\mathcal J} (n^{1/2})^{|\mathcal P|-|\mathcal  J|-1}. 
\end{eqnarray}
where   $C_1,C_2,\ldots $ are universal constants. 
Now note that since $x\in W_{n,\varepsilon}$,
we have a simple estimate 
\begin{equation}\label{L0.2}
n^{1/2}=n^\varepsilon n^{1/2-\varepsilon}\le n^\varepsilon 
|x^{(j_2)}-x^{(j_1)}|
\end{equation}
for any $j_1<j_2$.
Using (\ref{L0.1}) and (\ref{L0.2}), we obtain
\begin{align*}
&\mathbf E \left[ \left|\frac{\Delta(x+S_T)}{x^{(r)}-x^{(s)}+S^{(r)}_{T}-S^{(s)}_T}\right|;\,T\le n, A_n, B_{s,r}\right]\\
&\hspace{2cm}
\le \sum_{\mathcal J \subset \mathcal P\backslash(s,r)}
C_{\mathcal J} (n^{1/2})^{|\mathcal P|-|\mathcal  J|-1}\prod_{\mathcal  J} |x^{(i_2)}-x^{(i_1)}|\\
&\hspace{2cm}
\le \sum_{\mathcal J \subset \mathcal P\backslash(s,r)}\mathcal C_J( n^\varepsilon)^{|\mathcal P|-|\mathcal  J|-1}
 \prod_{\mathcal  J} |x^{(i_2)}-x^{(i_1)}|
  \prod_{\mathcal P\backslash (\mathcal  J\cup(s,r))} |x^{(j_2)}-x^{(j_1)}|\\
&\hspace{2cm} \le C_k n^{\varepsilon \frac{k(k-1)-1}{2}}\frac{\Delta(x)}{|x^{(r)}-x^{(s)}|}\le 
  C_{k} n^{\varepsilon \frac{k(k-1)}{2}}n^{-1/2}\Delta(x).
 \end{align*}
Thus,
\begin{equation}\label{L0.3}
E_1(x)\le \sum_{1\le s<r\le k} 
2{n^{1/2-\delta}}C_{k} n^{\varepsilon \frac{k(k-1)}{2}}n^{-1/2}\Delta(x)=
k(k-1) C_{k} n^{\varepsilon \frac{k(k-1)}{2}-\delta}\Delta(x).
\end{equation}

Now we estimate $E_2(x)$. 
Clearly,
$$
\overline A_n= \bigcup_{r=1}^k D_r,
$$
 where $D_r=\{\max_{1\le i\le n}|\xi_i^{(r)}|>n^{1/2-\delta}\}$. 
As in the first part of the proof,
\begin{eqnarray*}
\Delta (x+S_T)=\sum_{\mathcal J \subset \mathcal P}\prod_{\mathcal  J} (x^{(i_2)}-x^{(i_1)}) 
\prod_{\mathcal P\backslash\mathcal J} (S^{(j_2)}_T-S^{(j_1)}_T)
\end{eqnarray*}
and
\begin{eqnarray*}
\prod_{\mathcal P\backslash\mathcal J} (S^{(j_2)}_T-S^{(j_1)}_T)
=
\sum_{i_1,i_2,\ldots,i_k}\alpha^{\mathcal J}_{i_1,i_2,\ldots,i_k}(S^{(1)}_T)^{i_1}\ldots( S^{(k)}_T)^{i_k}.
\end{eqnarray*}
Then, using (\ref{DR}) once again, we get
\begin{align*}
&\mathbf E \left[\left|\prod_{\mathcal P\backslash\mathcal J} (S^{(j_2)}_T-S^{(j_1)}_T)\right|;T\leq n, D_r\right]\\ 
&\hspace{2cm}\le\sum_{i_1,i_2,\ldots, i_k}\left|\alpha^{\mathcal J}_{i_1,i_2,\ldots,i_k}\right|
C_{i_1}n^{i_1/2}\ldots\mathbf E \left[(M^{(r)}_{n})^{i_r};D_r\right]\ldots C_{i_k}n^{i_k/2}.
\end{align*}
Applying the following estimate, which will be proved at the end of the lemma,
\begin{equation}\label{L0.4}
\mathbf{E} \left[(M^{(r)}_{n})^{i_r};D_r\right]\leq C(\delta)n^{i_r/2-\alpha/2+1+(i_r+\alpha)\delta},
\end{equation}
we obtain
\begin{align*}
&\mathbf E \left[\left|\prod_{\mathcal P\backslash\mathcal J} (S^{(j_2)}_T-S^{(j_1)}_T)\right|;T\leq n, D_r\right]
\le C_{\mathcal J}C(\delta)(n^{1/2})^{|\mathcal P|-|\mathcal J|}\ n^{-\alpha/2+1+2\alpha\delta}.
\end{align*}
This implies that
\begin{align*}
&\mathbf{E}\left[|\Delta(x+S_T)|;\,T\leq n,D_r\right]\\
&\hspace{1cm}
\le C(\delta)n^{-\alpha/2+1+2\alpha\delta}\sum_{\mathcal J \subset \mathcal P}
C_{\mathcal J} (n^{1/2})^{|\mathcal P|-|\mathcal  J|}\prod_{\mathcal  J} |x^{(i_2)}-x^{(i_1)}|\\
&\hspace{1cm}
\le C(\delta)n^{-\alpha/2+1+2\alpha\delta}\sum_{\mathcal J \subset \mathcal P}\mathcal C_J( n^\varepsilon)^{|\mathcal P|-|\mathcal  J|}
 \prod_{\mathcal  J} |x^{(i_2)}-x^{(i_1)}|
  \prod_{\mathcal P\backslash \mathcal  J} |x^{(j_2)}-x^{(j_1)}|\\
&\hspace{1cm} \le C(\delta) n^{\varepsilon \frac{k(k-1)}{2}}n^{-\alpha/2+1+2\alpha\delta}\Delta(x).
\end{align*}
Consequently,
\begin{equation}
\label{L0.5}
E_2(x)\leq\sum_{r=1}^k\mathbf{E}\left[|\Delta(x+S_T)|;\,T\leq n,D_r\right]\leq
kC(\delta) n^{\varepsilon \frac{k(k-1)}{2}}n^{-\alpha/2+1+2\alpha\delta}\Delta(x).
\end{equation}
Applying (\ref{L0.3}) and (\ref{L0.5}) to the right hand side of (\ref{L0.0}), and choosing $\varepsilon$ and $\delta$ in an
appropriate way, we arrive at the conclusion. 

Thus, it remains to show (\ref{L0.4}).

It is easy to see that, for any $i_r\in(0,\alpha)$,
\begin{align*}
\mathbf{E} \left[(M^{(r)}_{n})^{i_r};D_r\right]&=i_r\int_0^\infty x^{i_r-1}\mathbf{P}(M^{(r)}_{n}>x, D_r)dx\\
&\leq n^{i_r(1/2+\delta)}\mathbf{P}(D_r)+i_r\int_{n^{1/2+\delta}}^\infty x^{i_r-1}\mathbf{P}(M^{(r)}_{n}>x)dx
\end{align*}
Putting $y=x/p$ in Corollary 1.11 of \cite{Nag79}, we get the inequality
$$
\mathbf{P}(|S^{(r)}_{n}|>x)\leq C(p)\Bigl(\frac{n}{x^2}\Bigr)^p+n\mathbf{P}(|\xi|>x/p).
$$
As was shown in \cite{Bor72}, this inequality remains valid for $M^{(r)}_{n}$, i.e.
$$
\mathbf{P}(M^{(r)}_{n}>x)\leq C(p)\Bigl(\frac{n}{x^2}\Bigr)^p+n\mathbf{P}(|\xi|>x/p).
$$
Using the latter bound with $p>i_r/2$,
we have
\begin{align*}
&\int_{n^{1/2+\delta}}^\infty x^{i_r-1}\mathbf{P}(M^{(r)}_{n}>x)dx\\
&\hspace{2cm}\leq
C(p) i_r n^p \int_{n^{1/2+\delta}}^\infty x^{i_r-1-2p}dx+n\int_{n^{1/2+\delta}}^\infty x^{i_r-1}\mathbf{P}(|\xi|>x/p)dx\\
&\hspace{2cm}\leq C(p) \frac{i_r}{2p-i_r}n^{p-(2p-i_r)(1/2+\delta)}+p^pn\mathbf{E}[|\xi|^{i_r},|\xi|>n^{1/2+\delta}/p]\\
&\hspace{2cm}\leq C(p)\Bigl(n^{p-(2p-i_r)(1/2+\delta)}+n^{1+(1/2+\delta)(i_r-\alpha)}\Bigr).
\end{align*}
Choosing $p>\alpha/2\delta$, we get
$$
\int_{n^{1/2+\delta}}^\infty x^{i_r-1}\mathbf{P}(M^{(r)}_{n}>x)dx\leq C(\delta)n^{i_r/2+1-\alpha/2}.
$$
Note that
\begin{equation}
 \label{L0.6}
\mathbf{P}(D_r)\leq n\mathbf{P}(|\xi|>n^{1/2-\delta})\leq C n^{1-\alpha(1/2-\delta)},
\end{equation}
we obtain
$$
\mathbf{E} \left[(M^{(r)}_{n})^{i_r};D_r\right]\leq C(\delta)n^{i_r/2+1-\alpha/2+(\alpha+i_r)\delta}.
$$
Thus, (\ref{L0.4}) is proved for $i_r\in(0,\alpha)$.
If $i_r=0$, then $\mathbf{E} \left[(M^{(r)}_{n})^{i_r};A_r\right]=\mathbf{P}(D_r)$. Therefore,
(\ref{L0.4}) with $i_r=0$ follows from (\ref{L0.6}).
\end{proof}
Define
$$
\nu_n:=\min\{k\geq1: x+S_k\in W_{n,\varepsilon}\}.
$$
\begin{lemma}
\label{lem1}
For every $\varepsilon>0$ holds
$$
\mathbf{P}(\nu_n>n^{1-\varepsilon})\leq\exp\{-Cn^{\varepsilon}\}.
$$
\end{lemma}
\begin{proof}
To shorten formulas in the proof we set $S_0=x$.
Also, set, for brevity, $b_n=[an^{1/2-\varepsilon}]$. 
The parameter $a$ will be chosen at the end of the proof.

First note that 
\begin{eqnarray*}
\{\nu_n>n^{1-\varepsilon}\}
\subset
\bigcap_{i=1}^{[n^{\varepsilon}/a^2]}
\bigcup_{1\leq j<l\leq k}\{|S^{(l)}_{i\cdot b_n^2}-S^{(j)}_{i\cdot b_n^2}|\leq n^{1/2-\varepsilon}\}.
\end{eqnarray*}
Then there exists at least one pair $\widehat j,\widehat l$ such that for at least at $[n^\varepsilon/(a^2k^2)]$ points 
$$
\mathcal I=\{i_1,\ldots, i_{[n^\varepsilon/(a^2k^2)]}\}\subset
\{b_n^2,2b_n^2,\ldots, [n^{\varepsilon}/a^2]b_n^2\}
$$
we have 
$$
|S^{(\widehat l)}_{i}-S^{(\widehat j)}_{i}|\leq
n^{1/2-\varepsilon} \mbox{ for } i\in \mathcal I.
$$
Without loss of generality we may assume that $\widehat j=1$ and $\widehat l=2$.
There should exist at least $[n^{\varepsilon}/(2a^2k^2)]$ points with the distance less than $2k^2b_n^2$. To simplify notation assume that points $i_1,\ldots i_{n^{\varepsilon}/(2a^2k^2)}$ enjoy this property:
$$
\max(i_2-i_1,i_3-i_2,\ldots, i_{[n^{\varepsilon}/(2a^2k^2)]}-i_{[n^{\varepsilon}/(2a^2k^2)]-1})\le 
2k^2b_n^2.
$$
In fact this means that $i_s-i_{s-1}$ can take only values 
$\{jn^{1-2\varepsilon},\ 1\le j\leq 2k^2\}$.
The above considerations imply that 
\begin{align*}
&\mathbf P\left(\nu_n>n^{1-\varepsilon}\right)\\
&\le {k\choose 2}{[n^\varepsilon/a^2]\choose [n^\varepsilon/(2a^2k^2]}
\mathbf P \left(|S^{( 2)}_{i}-S^{(1)}_{i}|
\leq n^{1/2-\varepsilon}\mbox{ for all }i\in\{i_1,\ldots,i_{[n^{\varepsilon}/(2a^2k^2)]}\}\right)\\
&\le {k\choose 2}{[n^\varepsilon/a^2]\choose [n^\varepsilon/(2a^2k^2]}
\prod_{s=2}^{[n^{\varepsilon}/(2a^2k^2)]}
\mathbf P \left(\left|(S^{(2)}_{i_s}-S^{(2)}_{i_s-1})-(S^{(1)}_{i_s}-S^{(1)}_{i_{s-1}})\right|
\leq 2n^{1/2-\varepsilon}\right).
\end{align*}
Using the Stirling formula, we get
$$
{[n^\varepsilon/a^2]\choose [n^\varepsilon/(2a^2k^2]}\leq 
\frac{a}{n^\varepsilon}(2k^2)^{n^\varepsilon/a^2}.
$$
By the Central Limit Theorem,
\begin{eqnarray*}
\lim_{n\to\infty}\mathbf P \left(\left|S^{(2)}_{jb_n^2}-S^{(1)}_{jb_n^2}\right|
\leq 2n^{1/2-\varepsilon}\right)=\int_{-\sqrt{2}/(a\sqrt{j})}^{\sqrt{2}/(a\sqrt{j})} \frac{1}{\sqrt{2\pi}}e^{-u^2/2}du
\leq\frac{2}{a}.
\end{eqnarray*}
Thus, for all sufficiently large $n$,
$$
\prod_{s=2}^{[n^{\varepsilon}/(2a^2k^2)]}
\mathbf P \left(\left|(S^{(2)}_{i_s}-S^{(2)}_{i_s-1})-(S^{(1)}_{i_s}-S^{(1)}_{i_{s-1}})\right|\right)
\leq \Bigl(\frac{4}{a}\Bigr)^{n^{\varepsilon}/(2a^2k^2)-1}
$$
Consequently,
$$
\mathbf P\left(\nu_n>n^{1-\varepsilon}\right)\leq
\left(\frac{4(2k^2)^{2k^2}}{a}\right)^{n^{\varepsilon}/(2a^2k^2)}
$$
Choosing $a=8(2k^2)^{2k^2}$, we complete the proof.
\end{proof}
\begin{lemma}\label{lem2}
For every $\varepsilon>0$ the inequality
$$
\mathbf{E}[|\Delta_t(x+S_n)|;\nu_n>n^{1-\varepsilon}]\leq c_t \Delta_1(x)\exp\{-Cn^\varepsilon\}
$$
holds.
\end{lemma}
\begin{remark}
If $\mathbf{E}|\xi|^\alpha<\infty$ for some $\alpha>k-1$, then the claim in the lemma follows easily from the
H{\"o}lder inequality and Lemma~\ref{lem1}. But our moment assumption requires more
detailed analysis.\hfill $\diamond$
\end{remark}
\begin{proof}
We give the proof only for $t=0$.

For $1\le l<i\le k$ define
$$
G_{l,i}=\left\{|x^{(l)}-x^{(i)}+S^{(l)}_{jb_n^2}-S^{(i)}_{jb_n^2}|\leq n^{1/2-\varepsilon}\mbox{\,for at least\,}
\left[\frac{n^\varepsilon}{a^2k^2}\right]
\mbox{\,values of\,}j\leq\frac{n^\varepsilon}{a^2}\right\}.
$$
Noting that $\{\nu_n>n^{1-\varepsilon}\}\subset\bigcup G_{l,i}$, we get
$$
\mathbf{E}[|\Delta(x+S_n)|;\nu_n>n^{1-\varepsilon}]\leq 
{k\choose 2}\mathbf{E}[|\Delta(x+S_n)|;G_{1,2}].
$$
Therefore, we need to derive an upper bound for $\mathbf{E}[|\Delta(x+S_n)|;G_{1,2}]$.

Let $\mu=\mu_{1,2}$ be the moment when $|x^{(2)}-x^{(1)}+S^{(2)}_{jb_n^2}-S^{(1)}_{jb_n^2}|\leq n^{1/2-\varepsilon}$ for the
$[n^\varepsilon/(a^2k^2)]$ time. Then it follows from the proof of the previous lemma that
\begin{equation}
\label{L2.1}
\mathbf{P}(\mu\leq n^{1-\varepsilon})=\mathbf{P}(G_{1,2})\leq \exp\{-Cn^\varepsilon\}.
\end{equation}
Using the inequality $|a+b|\le (1+|a|)(1+|b|)$ one can see that 
\begin{align}\label{L2.2}
\nonumber
\mathbf{E}[|\Delta(x+S_n)|;G_{1,2}]
&\leq\mathbf{E}[|\Delta(x+S_n)|;\mu\leq n^{1-\varepsilon}]
=\sum_{m=1}^{n^{1-\varepsilon}}
\mathbf{E}[|\Delta(x+S_n)|;\mu=m]
\\\nonumber
&\le 
\sum_{m=1}^{n^{1-\varepsilon}}
\mathbf{E}[\Delta_1(S_n-S_m)]\mathbf{E}[\Delta_1(x+S_m);\mu=m]\\
&\le \max_{m\le n^{1-\varepsilon}} \mathbf{E}[\Delta_1(S_n-S_m)]
\mathbf{E}[\Delta_1(x+S_\mu);\mu\leq n^{1-\varepsilon}].
\end{align}
Making use of (\ref{DR}), one can verify that
\begin{equation}\label{L2.3}
\max_{m\le n^{1-\varepsilon}} \mathbf{E}[\Delta_1(S_n-S_m)]\leq Cn^{k(k-1)/4}.
\end{equation}

Recall that by the definition of $\mu$ we have  $|x^{(2)}-x^{(1)}+S^{(2)}_\mu-S^{(1)}_\mu|\leq n^{1/2-\varepsilon}.$ 
Therefore,
\begin{align*}
\Delta_1(x+S_\mu)&\le n^{1/2-\varepsilon}\frac{\Delta_1(x+S_\mu)}{1+|x^{(2)}-x^{(1)}+S^{(2)}_\mu-S^{(1)}_\mu|}\\
&\le n^{1/2-\varepsilon}\frac{\Delta_2(x)}{2+|x^{(2)}-x^{(1)}|}\frac{\Delta_2(S_\mu)}{2+|S^{(2)}_\mu-S^{(1)}_\mu|}
 \end{align*}
It is easy to see that
$$
\frac{\Delta_2(S_\mu)}{2+|S^{(2)}_\mu-S^{(1)}_\mu|}\leq \sum_{i_1,\ldots,i_k}C_{(i_1,\ldots,i_k)}\prod\left(|S_\mu^{(r)}|\right)^{i_r},
$$
where the sum is taken over all $i_1,\ldots,i_k$ such that all $i_1,i_2\le k-2, i_3,\ldots i_k\leq k-1$, there is at most one
$i_j=k-1$, and the sum $\sum i_r$ does not exceed  $k(k-1)/2$. Thus,
\begin{align*}
&\mathbf{E}\left[\left|\frac{\Delta_2(S_\mu)}{2+|S^{(2)}_\mu-S^{(1)}_\mu|}\right|;\mu\leq n^{1-\varepsilon}\right]
\leq \sum_{i_1,\ldots,i_k} C_{(i_1,\ldots,i_k)}
\mathbf{E}\left[\prod_{r=1}^k\left(|S_\mu^{(r)}|\right)^{i_r};\mu\leq n^{1-\varepsilon}\right]\\
&\hspace{0.5cm}\leq \sum_{i_1,\ldots,i_k} C_{(i_1,\ldots,i_k)}
\mathbf{E}\left[\left(|S_\mu^{(1)}|\right)^{i_1}\left(|S_\mu^{(2)}|\right)^{i_2};\mu\leq n^{1-\varepsilon}\right]
\prod_{r=3}^k\mathbf{E}\left(M_n^{(r)}\right)^{i_r}.
\end{align*}
Since $i_1\le k-2$ and $i_2\le k-2$, we can apply  the H{\"o}lder inequality, which gives
$$
\mathbf{E}\left[\left(|S_\mu^{(1)}|\right)^{i_1}\left(|S_\mu^{(2)}|\right)^{i_2};\mu\leq n^{1-\varepsilon}\right]
\leq n^{(i_1+i_2)/2}\exp\{-Cn^\varepsilon\}.
$$
Consequently,
\begin{equation}\label{L2.4}
\mathbf{E}[|\Delta(x+S_\mu)|;\mu\leq n^{1-\varepsilon}]\leq
c \Delta_2(x)n^{k(k-1)/2}\exp\{-Cn^\varepsilon\}.
\end{equation}
Plugging (\ref{L2.3}) and (\ref{L2.4}) into (\ref{L2.2}), we arrived at the conclusion.
\end{proof}
\begin{lemma}\label{lem3}
There exists a constant $C$ such that
$$
\mathbf{E}[\Delta(x+S_n);T>n]\leq C\Delta_1(x)
$$
for all $n\geq1$ and all $x\in W$.
\end{lemma}
\begin{proof}
We first split the  expectation into 2 parts,
\begin{align*}
&\mathbf E [\Delta(x+S_n);T>n]=E_1(x)+E_2(x)\\
&\hspace{1cm}=\mathbf E\left[\Delta(x+S_n);T>n,\nu_n\leq n^{1-\varepsilon}\right]
+\mathbf E\left[\Delta(x+S_n);T>n,\nu_n> n^{1-\varepsilon}\right].
\end{align*}
By Lemma~\ref{lem2}, the second term on the right hand side is bounded by
$$
E_2(x)\le c\Delta_1(x)\exp\{-Cn^\varepsilon\}.
$$ 
Using Lemma~\ref{lem-1}, we have
\begin{align*}
 E_1(x)&\le\sum_{i=1}^{n^{1-\varepsilon}}
\int_{W_{n,\varepsilon}}\mathbf P\{\nu_n=k,T>k, x+S_k\in dy\}
\mathbf E[\Delta(y+S_{n-k});T>n-k]\\
&\le\sum_{i=1}^{n^{1-\varepsilon}}
\int_{W_{n,\varepsilon}}\mathbf P\{\nu_n=k,T>k, x+S_k\in dy\}
\mathbf E[\Delta(y+S_{n});T>n]\\
&=\sum_{i=1}^{n^{1-\varepsilon}}
\int_{W_{n,\varepsilon}}\mathbf P\{\nu_n=k,T>k, x+S_k\in dy\}
\left(\Delta(y)-\mathbf E[\Delta(y+S_{T});T\leq n]\right),
\end{align*}
in the last step we used the fact that $\Delta(x+S_n)$ is a martingale.
Then, by Lemma~\ref{lem0},
\begin{eqnarray*}
E_1(x)&\le &\left(1+\frac{C}{n^\gamma}\right)
\sum_{i=1}^{n^{1-\varepsilon}}
\int_{W_{n,\varepsilon}}\mathbf P\{\nu_n=k,T>k, x+S_k\in dy\}\Delta(y)\\
&\le &\left(1+\frac{C}{n^\gamma}\right)
\mathbf E[\Delta(x+S_{\nu_n}); \nu_n\le n^{1-\varepsilon},T>\nu_n].
\end{eqnarray*}
Using Lemma~\ref{lem-1} once again, we arrive at the bound
$$
E_1(x)\le \left(1+\frac{C}{n^\gamma}\right)
\mathbf E[\Delta(x+S_{n^{1-\varepsilon}}); T>n^{1-\varepsilon}].
$$
As a result we have
\begin{align}\label{L3.1}
\nonumber
&\mathbf E [\Delta(x+S_n);T>n]\\
&\hspace{1cm}\leq
\left(1+\frac{C}{n^\gamma}\right)
\mathbf E[\Delta(x+S_{n^{1-\varepsilon}}); T>n^{1-\varepsilon}]
+c\Delta_1(x)\exp\{-Cn^\varepsilon\}.
\end{align}
Iterating this procedure $m$ times, we obtain
\begin{align}\label{L3.2}
\nonumber
&\mathbf E [\Delta(x+S_n);T>n]\leq
\prod_{j=0}^m\left(1+\frac{C}{n^{\gamma(1-\varepsilon)^j}}\right)\times\\
&\hspace{0.1cm}\left(\mathbf E[\Delta(x+S_{n^{(1-\varepsilon)^{m+1}}}); T>n^{(1-\varepsilon)^{m+1}}]
+c\Delta_1(x)\sum_{j=0}^m\exp\{-Cn^{\varepsilon(1-\varepsilon)^j}\}\right).
\end{align}
Choosing $m=m(n)$ such that $n^{(1-\varepsilon)^{m+1}}\leq10$ and noting that the product and the sum remain 
uniformly bounded, we finish the proof of the lemma.
\end{proof}
\begin{lemma}\label{lem33}
The function $V^{(T)}(x):=\lim_{n\to\infty}\mathbf{E}[\Delta(x+S_n);T>n]$ has the following properties:
\begin{equation}
\label{L33.1}
\Delta(x)\leq V^{(T)}(x)\leq C\Delta_1(x)
\end{equation}
and
\begin{equation}
\label{L33.2}
V^{(T)}(x)\sim\Delta(x)\quad\text{if}\quad\min_{j<k}(x^{(j+1)}-x^{(j)})\to\infty.
\end{equation}
\end{lemma}
\begin{proof}
Since $\Delta(x+S_n){\bf 1}\{T_x>n\}$ is a submartingale, the limit $\lim_{n\to\infty}\mathbf{E}[\Delta(x+S_n);T>n]$
exists, and the function $V^{(T)}$ satisfies $V^{(T)}(x)\geq \Delta(x),\quad x\in\{y:\Delta(y)>0\}$. The upper bound
in (\ref{L33.1}) follows immediately from Lemma \ref{lem3}.

To show (\ref{L33.2}) it suffices to obtain an upper bound of the form $(1+o(1))\Delta(x)$.
Furthermore, because of monotonicity of $\mathbf{E}[\Delta(x+S_n);T>n]$, we can get such a bound for a specially chosen
subsequence $\{n_m\}$. Choose $\varepsilon$ so that (\ref{L3.2}) is valid, and set $n_m=(n_0)^{(1-\varepsilon)^{-m}}$.
Then we can rewrite (\ref{L3.2}) in the following form
\begin{align*}
&\mathbf E [\Delta(x+S_{n_m});T>n_m]\leq\\
&\hspace{1cm}\prod_{j=0}^{m-1}\left(1+\frac{C}{n_j^\gamma}\right)\times
\left(\mathbf E[\Delta(x+S_{n_0}); T>n_0]
+c\Delta_1(x)\sum_{j=0}^{m-1}\exp\{-Cn_j^{\varepsilon}\}\right).
\end{align*}
It is clear that for every $\delta>0$ we can choose $n_0$ such that
$$
\prod_{j=0}^{m-1}\left(1+\frac{C}{n_j^\gamma}\right)\leq 1+\delta
\quad\text{and}\quad
\sum_{j=0}^{m-1}\exp\{-Cn_j^{\varepsilon}\}\leq\delta
$$
for all $m\geq1$. Consequently,
$$
V^{(T)}(x)=\lim_{m\to\infty}\mathbf E [\Delta(x+S_{n_m});T>n_m]
\leq (1+\delta)\mathbf E[\Delta(x+S_{n_0}); T>n_0]+C\delta\Delta_1(x).
$$
It remains to note that $E[\Delta(x+S_{n_0}); T>n_0]\sim\Delta(x)$ and that
$\Delta_1(x)\sim\Delta(x)$ as $\min_{j<k}(x^{(j+1)}-x^{(j)})\to\infty$.
\end{proof}

%%%%%%%%%%%%%%%%%%%%%%%%%%%%%%%%%%%%%%%%%%%%%%%%%%%%%%%%%%%%%%%%%%%%%%%%%%%%%%%%%%%%%%%%%%%%%%%%%%%%%%%%%%%%%%%%%%%

\subsection{Proof of Proposition~\ref{prop1}}
We start by showing that Lemma ~\ref{lem3} implies the integrability of
 $\Delta(x+S_{\tau_x})$. Indeed, setting $\tau_x(n):=\min\{\tau_x,n\}$
and $T_x(n):=\min\{T_x,n\}$, and using the fact that $|\Delta(x+S_n)|$ is a submartingale, we have
\begin{align*}
\mathbf{E}|\Delta(x+S_{\tau_x(n)})|&\leq \mathbf{E}|\Delta(x+S_{T_x(n)})|\\
&=
\mathbf{E}[\Delta(x+S_n){\rm 1}\{T_x(n)>n\}]-\mathbf{E}[\Delta(x+S_{T_x}){\rm 1}\{T_x\le n\}].
\end{align*}
Since $\Delta(x+S_n)$ is a martingale, we have  
$$
\mathbf{E}[\Delta(x+S_T){\rm 1}\{T\le n\}]=
\mathbf{E}[\Delta(x+S_n){\rm 1}\{T\le n\}]=\Delta(x)-\mathbf{E}[\Delta(x+S_n){\rm 1}\{T>n\}].
$$
Therefore, we get
$$
\mathbf{E}|\Delta(x+S_{\tau_n})|\leq 2\mathbf{E}[\Delta(x+S_n){\rm 1}\{T>n\}]-\Delta(x).
$$
This, together with Lemma~\ref{lem3}, implies that the sequence
$\mathbf{E}[|\Delta(x+S_{\tau})|\mathbf 1\{\tau\le n\}]$ is uniformly bounded. Then, the
finiteness of the expectation
$\mathbf{E}|\Delta(x+S_\tau)|$ follows from the monotone convergence. 

To prove (a) note that since $\Delta(x+S_n)$ is a martingale, we have an equality 
$$
\mathbf{E}[\Delta(x+S_n);\tau_x>n]
=
\Delta(x)-\mathbf{E}[\Delta(x+S_n);\tau_x\le n]
=
\Delta(x)-\mathbf{E}[\Delta(x+S_{\tau_x});\tau_x\le n].
$$
Letting $n$ to infinity we obtain (a) by the dominated convergence theorem.

For (b) note that 
$$
\Delta(x+S_n) {\bf 1}\{\tau_x>n\}\le \Delta(y+S_n) {\bf 1}\{\tau_x>n\}\le \Delta(y+S_n) {\bf 1}\{\tau_y>n\}.
$$
Then letting $n$ to infinity and applying (a) we obtain (b).

(c) follows directly from Lemma~\ref{lem3}.

We now turn to the proof of (d). 
It follows from (\ref{L33.2}) and the inequality $\tau_x\leq T_x$ that
$$
V(x)\leq V^{(T)}(x)\leq (1+o(1))\Delta(x).
$$
Thus, we need to get a lower bound of the form $(1+o(1))\Delta(x)$.
We first note that 
$$
V(x)=\Delta(x)-\mathbf{E}[\Delta(x+S_{\tau_x})]\geq\Delta(x)-\mathbf{E}[\Delta(x+S_{\tau_x});T_x>\tau_x].
$$
Therefore, it is sufficient to show that 
\begin{equation}
\label{suff}
\mathbf{E}[\Delta(x+S_{\tau_x});T_x>\tau_x]=o(\Delta(x))
\end{equation}
under the condition $\min_{j<k}(x^{(j+1)}-x^{(j)})\to\infty$.

The sequence $V^{(T)}(x+S_n){\bf 1}\{T_x>n\}$ is a non-negative martingale. 
Then, arguing as in Lemma~\ref{lem-1}, one can easily see
that $V^{(T)}(x+S_n){\bf 1}\{\tau_x>n\}$ is a supermartingale.

We bound $\mathbf E [V^{(T)}(x+S_n){\bf 1}\{\tau_x>n\}]$ from below using its supermartingale property.  This is similar to the Lemma~\ref{lem3}, where an upper bound has been obtained using submartingale properties of  $\Delta(x+S_n){\bf 1}\{T_x>n\}$. We have
\begin{align*}
&\mathbf{E}[V^{(T)}(x+S_n);\tau_x>n]\\
&\ge\sum_{i=1}^{n^{1-\varepsilon}}
\int_{W_{n,\varepsilon}}\mathbf P\{\nu_n=k,\tau_x>k, x+S_k\in dy\}
\mathbf E[V^{(T)}(y+S_{n-k});\tau_y>n-k]\\
&\ge\sum_{i=1}^{n^{1-\varepsilon}}
\int_{W_{n,\varepsilon}}\mathbf P\{\nu_n=k,\tau_x>k, x+S_k\in dy\}
\mathbf E[V^{(T)}(y+S_{n});\tau_y>n]\\
&=\sum_{i=1}^{n^{1-\varepsilon}}
\int_{W_{n,\varepsilon}}\hspace{-0.3cm}\mathbf P\{\nu_n=k,\tau_x>k, x+S_k\in dy\}
\left(V^{(T)}(y)-\mathbf E[V^{(T)}(y+S_{\tau_y});\tau_y\leq n]\right).
\end{align*}
Then, applying (\ref{L33.1}) and (\ref{eq00}), we obtain
$$
\mathbf{E}[V^{(T)}(x+S_n);\tau_x>n]\geq
\left(1-\frac{C}{n^\gamma}\right)\mathbf{E}[V^{(T)}(x+S_{n^{1-\varepsilon}});\tau_x>n^{1-\varepsilon},\nu_n\leq n^{1-\varepsilon}].
$$
Using now Lemma~\ref{lem2}, we have
$$
\mathbf{E}[V^{(T)}(x+S_n);\tau_x>n]\geq
\left(1-\frac{C}{n^\gamma}\right)\mathbf{E}[V^{(T)}(x+S_{n^{1-\varepsilon}});\tau_x>n^{1-\varepsilon}]
-C\Delta_1(x)e^{-Cn^\varepsilon}.
$$
Starting from $n_0$ and iterating this procedure, we obtain for the sequence $n_m=(n_0)^{(1-\varepsilon)^{-m}}$
the inequality
\begin{align*}
\mathbf E [V^{(T)}(x+S_{n_m});\tau_x>n_m]
\geq
\prod_{j=1}^m\left(1-\frac{C}{n_0^{\gamma(1-\varepsilon)^{-j}}}\right)
\mathbf E[V^{(T)}(x+S_{n_0}); \tau_x>n_0]\\
-c\Delta_1(x)\sum_{j=1}^m\exp\{-Cn_0^{\varepsilon(1-\varepsilon)^{-j}}\}.
\end{align*}
Next we fix a constant $\delta>0$ and pick $n_0$ such that 
$$
\prod_{j=1}^\infty\left(1-\frac{C}{n_0^{\gamma(1-\varepsilon)^{-j}}}\right)\ge (1-\delta),\quad 
c\sum_{j=1}^\infty\exp\{-Cn_0^{\varepsilon(1-\varepsilon)^{-j}}\}\le \delta. 
$$
This is possible since  both the series and the product converge. 
Together with the fact that  $V^{(T)}(x+S_n){\bf 1}\{\tau_x>n\}$ is a supermartingale 
and the with lower bound in (\ref{L33.1})  this gives us,
$$
\lim_{n\to\infty}\mathbf E [V^{(T)}(x+S_n);\tau_x>n]\ge (1-\delta)\mathbf E[\Delta(x+S_{n_0}); \tau_x>n_0]-\delta\Delta_1(x).
$$
As is not difficult to see  $\mathbf E[\Delta(x+S_{n_0}); \tau_x>n_0]\sim \Delta(x)$ and $\Delta_1(x)\sim\Delta(x)$ 
as $\min_{2\leq j\leq k}(x^{(j)}-x^{(j-1)})\to \infty$. Therefore, since $\delta>0$ is arbitrary we have a lower asymptotic bound 
\begin{equation}\label{P3}
\lim_{n\to\infty}\mathbf E [V^{(T)}(x+S_n);\tau_x>n]\ge(1-o(1))\Delta(x), 
\end{equation}
provided that $\min_{2\leq j\leq k}(x^{(j)}-x^{(j-1)})\to\infty$.

Using the martingale property of $V^{(T)}(x+S_n){\bf 1}\{T_x>n\}$ and noting that
$$
\{T_x>n\}=\{\tau_x>n\}\bigcup\left(\bigcup_{k=1}^n\{T_x>n,\,\tau_x=k\}\right),
$$ 
we get
\begin{align*}
V^{(T)}(x)&=\mathbf{E}[V^{(T)}(x+S_n){\bf 1}\{T_x>n\}]\\
&=\mathbf{E}[V^{(T)}(x+S_n);\tau_x>n]+\sum_{k=1}^n\mathbf{E}[V^{(T)}(x+S_n){\bf 1}\{T_x>n\};\tau_x=k]\\
&=\mathbf{E}[V^{(T)}(x+S_n);\tau_x>n]+\sum_{k=1}^n\mathbf{E}[V^{(T)}(x+S_k){\bf 1}\{T_x>k\};\tau_x=k]\\
&=\mathbf{E}[V^{(T)}(x+S_n);\tau_x>n]+\mathbf{E}[V^{(T)}(x+S_{\tau_x}){\bf 1}\{T_x>\tau_x\};\tau_x\leq n].
\end{align*}
Letting $n\to\infty$, we obtain
\begin{equation}\label{P4}
\lim_{n\to\infty}\mathbf E [V^{(T)}(x+S_n);\tau_x>n]=V^{(T)}(x)-
\mathbf E [V^{(T)}(x+S_{\tau_x});T_x>\tau_x].
\end{equation}
Combining (\ref{L33.2}), (\ref{P3}) and (\ref{P4}), we have $E [V^{(T)}(x+S_{\tau_x});T_x>\tau_x]=o(\Delta(x))$.
Now (\ref{suff}) follows from the obvious bound
$$
\mathbf E [\Delta(x+S_{\tau_x});T_x>\tau_x]\le \mathbf E [V^{(T)}(x+S_{\tau_x});T_x>\tau_x].
$$
Thus, the proof of (d) is finished.

To prove (e) note that it follows from (d) that there exists $R$ and $\delta>0$ such that 
$V(x)\ge \delta$ on the set $S_R=\{x: \min_{2\leq j\leq k}(x^{(j)}-x^{(j-1)})>R\}.$ Then, with a positive probability $p$ the 
random walk can reach this set after $N$ steps if $N$ is sufficiently large. Therefore, 
\begin{align*}
V(x)=\sup_{n\geq1}\mathbf{E}[\Delta(x+S_n);\tau_x>n]\ge \int_{S_R}\mathbf P\{x+S_N\in dy\}\sup_{n\geq1}\mathbf{E}[\Delta(y+S_n);\tau_y>n]
\\
= \int_{S_R}\mathbf P\{x+S_N\in dy\}V(y)\ge \delta p>0.
\end{align*}
This completes the proof of the proposition. 

%%%%%%%%%%%%%%%%%%%%%%%%%%%%%%%%%%%%%%%%%%%%%%%%%%%%%%%%%%%%%%%%%%%%%%%%%%%%%%%%%%%%%%%%%%%%%%%%%%%%%%%%%%%%%%%%%%%
%%%%%%%%%%%%%%%%%%%%%%%%%%%%%%%%%%%%%%%%%%%%%%%%%%%%%%%%%%%%%%%%%%%%%%%%%%%%%%%%%%%%%%%%%%%%%%%%%%%%%%%%%%%%%%%%%%%
%%%%%%%%%%%%%%%%%%%%%%%%%%%%%%%%%%%%%%%%%%%%%%%%%%%%%%%%%%%%%%%%%%%%%%%%%%%%%%%%%%%%%%%%%%%%%%%%%%%%%%%%%%%%%%%%%%%
\section{Coupling}
We start by formulating a classical result on the normal approximation of random walks.
\begin{lemma}
\label{lem4}
If $\mathbf{E}\xi^{2+\delta}<\infty$ for some $\delta\in(0,1)$, then one can define a Brownian motion $B_t$ 
on the same probability space such that, for any $a$ satisfying $0<a<\frac{\delta}{2(2+\delta)}$,
\begin{equation}\label{L4}
\mathbf{P}\left(\sup_{u\leq n}|S_{[u]}-B_{u}|\geq n^{1/2-a}\right)=o\left(n^{2a+a\delta-\delta/2}\right).
\end{equation}
\end{lemma}
This statement easily follows from Theorem~2 of \cite{M76}, see also Theorem~2 of \cite{Bor83}.

\begin{lemma}
\label{lem6}
There exists a finite constant $C$ such that
\begin{equation}\label{L6.1}
\mathbf{P}(\tau^{bm}_{y}>n)\leq C\frac{\Delta(y)}{n^{k(k-1)/4}},\quad y\in W.
\end{equation}
Moreover,
\begin{equation}\label{L6.2}
\mathbf{P}(\tau^{bm}_{y}>n)\sim \varkappa\frac{\Delta(y)}{n^{k(k-1)/4}},\end{equation}
uniformly in $y\in W$ satisfying $|y|\le \theta_n\sqrt{n}$ with some $\theta_n\to0$.  
Finally, the density $b_{t}(y,z)$ of the probability 
$
\mathbf{P}(\tau^{bm}_{y}>t, B_t\in dz)
$ is 
\begin{equation}\label{L6.3}
b_t(y,z)\sim K t^{-k/2 }
e^{-|z|^2/(2t)}\Delta(y)\Delta(z) t^{-\frac{k(k-1)}{2}}
\end{equation}
uniformly in $y,z \in W$ satisfying $|y|\le \theta_n\sqrt{n}$  and  $|z|\le \sqrt {n/\theta_n}$ with some $\theta_n\to0$. 
Here,
\begin{align*}
K=(2\pi )^{-k/2 }\prod_{l=0}^{k-1}\frac{1}{l!};
\ \varkappa=K\frac{1}{k!}\int_{\mathbf R^k}e^{-|x|^2/2}|\Delta(x)|dx=
K\frac{1}{k!}
2^{3k/2}\prod_{j=1}^k\Gamma(1+j/2).
\end{align*}
\end{lemma}
\begin{proof}
(\ref{L6.1}) has been proved by Varopoulos \cite{Var99}, see Theorem 1 and formula (0.4.1) there. The proof of (\ref{L6.2}) and (\ref{L6.3}) can be found in Sections 5.1-5.2 of \cite{Grab99}. 
\end{proof}

Using the coupling we can translate the results of Lemma~\ref{lem5} to 
the random walks setting when $y\in W_{n,\varepsilon}$.
\begin{lemma}\label{lem5}
For all sufficiently small $\varepsilon>0$,
\begin{equation}\label{L6.4}
\mathbf{P}(\tau_y>n)=\varkappa\Delta(y)n^{-k(k-1)/4}(1+o(1)),\quad\text{as }n\to\infty
\end{equation}
uniformly in $y\in W_{n,\varepsilon}$ such that $|y|\le \theta_n \sqrt n$ for 
some $\theta_n\to 0$. Moreover, there exists a constant $C$ such that  
\begin{equation}\label{L6.5}
\mathbf{P}(\tau_y>n)\le C \Delta(y)n^{-k(k-1)/4},
\end{equation}
uniformly in $y\in W_{n,\varepsilon},n\ge 1$. Finally, 
for any bounded open set $D\subset W$,
\begin{equation}\label{L6.6}
 \mathbf P(\tau_y>n, y+S_n\in \sqrt n D)\sim K\Delta(y) n^{-k(k-1)/4} 
\int_D dz e^{-|z|^2/2}\Delta(z).
\end{equation}

\end{lemma}
\begin{proof}
For every $y\in W_{n,\varepsilon}$ denote
$$
y^\pm=(y_i\pm(i-1)n^{1/2-2\varepsilon},1\le i\le k).
$$
Define $A=\left\{\sup_{u\leq n}|S^{(r)}_{[u]}-B^{(r)}_{u}|\le n^{1/2-2\varepsilon}\text{ for all }r\leq k\right\}$,
where $B^{(r)}$ are as in Lemma~\ref{lem4}.
Then, using (\ref{L4}) with $a=2\varepsilon$, we obtain
\begin{align}\label{L5.1}
\nonumber
\mathbf{P}(\tau_y>n)&=\mathbf{P}(\tau_y>n,A)+o\left(n^{-r}\right)\\
\nonumber
&=\mathbf{P}(\tau_y>n,\tau^{bm}_{y^+}>n,A)+o\left(n^{-r}\right)\\
\nonumber
&\leq \mathbf{P}(\tau^{bm}_{y^+}>n,A)+o\left(n^{-r}\right)\\
&=\mathbf{P}(\tau^{bm}_{y^+}>n)+o\left(n^{-r}\right),
\end{align}
where $r=r(\delta,\varepsilon)=\delta/2-4\varepsilon-2\varepsilon\delta.$
In the same way one can get
\begin{equation}
\label{L5.2}
\mathbf{P}(\tau^{bm}_{y^-}>n)\leq \mathbf{P}(\tau_{y}>n)+o\left(n^{-r}\right).
\end{equation}
By Lemma~\ref{lem6}, 
$$
\mathbf P(\tau^{bm}_{y^\pm}>n)\sim \varkappa\Delta(y^\pm)n^{-k(k-1)/4}.
$$
Next, since $y\in W_{n\varepsilon}$,
$$
\Delta(y^\pm)=\Delta(y)(1+O(n^{-\varepsilon}))
$$ 
Therefore, we conclude that
$$
\mathbf{P}(\tau^{bm}_{y^\pm}>n)=\varkappa\Delta(y)n^{-k(k-1)/4}(1+O(n^{-\varepsilon})).
$$
{F}rom this relation and bounds (\ref{L5.1}) and (\ref{L5.2}) we obtain
$$
\mathbf{P}(\tau_{y}>n)=\varkappa\Delta(y)n^{-k(k-1)/4}(1+O(n^{-\varepsilon}))+o\left(n^{-r}\right).
$$
Thus, it remains to show that
\begin{equation}
\label{L5.3}
n^{-r}=o(\Delta(y)n^{-k(k-1)/4})
\end{equation}
for all sufficiently small $\varepsilon>0$ and all $y\in W_{n,\varepsilon}$.
For that note that for $y\in W_{n,\varepsilon}$,
$$
\Delta(y)n^{-k(k-1)/4}\ge \prod_{i<j}(j-i)n^{-\varepsilon\frac{k(k-1)}{2}}.
$$
Therefore, (\ref{L5.3}) will be valid for all $\varepsilon$ satisfying
$$
r=4\varepsilon+2\delta\varepsilon-\delta/2<-\varepsilon\frac{k(k-1)}{2}.
$$
This proves (\ref{L6.4}). 
To prove  (\ref{L6.5}) it is sufficient to substitute (\ref{L6.1}) in (\ref{L5.1}).

The proof of  (\ref{L6.6}) is similar. Define two sets, 
\begin{align*}
D^+=\{z\in W: dist(z, D)\le 4k n^{-2\varepsilon}\},\ 
D^-=\{z\in D: dist(z, \partial D)\le 4k n^{-2\varepsilon}\}.
\end{align*}
Clearly $D^-\subset D\subset D^+.$ Then, arguing as above, we get
\begin{align}\label{L6.7}
\nonumber
\mathbf{P}(\tau_y>n,y+S_n\in \sqrt n D)&\le 
\mathbf{P}(\tau_y>n,y+S_n\in \sqrt n D, A)+o\left(n^{-r}\right)\\
\nonumber
&\le 
\mathbf{P}(\tau^{bm}_{y^+}>n,y^++B_n\in \sqrt n D^+, A)+o\left(n^{-r}\right)\\
&\le 
\mathbf{P}(\tau^{bm}_{y^+}>n,y^++B_n\in \sqrt n D^+)+o\left(n^{-r}\right).
\end{align}
Similarly,
\begin{equation}\label{L6.8}
\mathbf{P}(\tau_y>n,y+S_n\in \sqrt n D)\ge 
\mathbf{P}(\tau^{bm}_{y^-}>n,y^-+B_n\in \sqrt n D^-)+o\left(n^{-r}\right).
\end{equation}
Now we apply (\ref{L6.3}) and obtain 
\begin{align*}
 \mathbf{P}(\tau^{bm}_{y^\pm}>n,y^\pm+B_n\in \sqrt n D^\pm)&\sim 
K \Delta(y^\pm) \int_{\sqrt n D^\pm} dz e^{-|z|^2/(2n)}\Delta(z) n^{-\frac{k}{2}}n^{-\frac{k(k-1)}{4}} \\
&=K \Delta(y^\pm) \int_{ D^\pm} dz e^{-|z|^2/2}\Delta(z) n^{-\frac{k(k-1)}{4}}.
\end{align*}
It is sufficient to note now that 
$$
\Delta(y^\pm)\sim\Delta(y) \mbox{ and } \int_{ D^\pm} dz e^{-|z|^2/2}\Delta(z)\to 
\int_{ D} dz e^{-|z|^2/2}\Delta(z)
$$
as $n\to \infty$. From these relations and bounds (\ref{L6.7}) and 
(\ref{L6.8})  we obtain
$$
\mathbf{P}(\tau_y>n,y+S_n\in \sqrt n D)=
(K+o(1)) \Delta(y) \int_{ D} dz e^{-|z|^2/2}\Delta(z) n^{-\frac{k(k-1)}{4}}
+o\left(n^{-r}\right). 
$$
Recalling (\ref{L5.3}) we arrive at the conclusion. 
\end{proof}
%%%%%%%%%%%%%%%%%%%%%%%%%%%%%%%%%%%%%%%%%%%%%%%%%%%%%%%%%%%%%%%%%%%%%%%%%%%%%%%%%%%%%%%%%%%%%%%%%%%%%%%%%%%%%%%%%%%
\section{Asymptotics for $\mathbf P\{\tau_x>n\}$}\label{sect.asymptotics}
We first note that, in view of Lemma~\ref{lem1},
\begin{align}\label{T1.1}
\nonumber
\mathbf{P}(\tau_x>n)&=\mathbf{P}(\tau_x>n,\nu_n\leq n^{1-\varepsilon})+\mathbf{P}(\tau_x>n,\nu_n> n^{1-\varepsilon})\\
&=\mathbf{P}(\tau_x>n,\nu_n\leq n^{1-\varepsilon})+O\left(e^{-Cn^\varepsilon}\right).
\end{align}
Using the strong Markov property, we get for the first term the following estimates
\begin{align}
\label{T1.2}
\nonumber
&\int_{W_{n,\varepsilon}}\mathbf{P}\left(S_{\nu_n}\in dy,\tau_x>\nu_n,\nu_n\leq n^{1-\varepsilon}\right)\mathbf{P}(\tau_y>n)
\leq\mathbf{P}(\tau_x>n,\nu_n\leq n^{1-\varepsilon})\\
&\hspace{1cm}\leq\int_{W_{n,\varepsilon}}\mathbf{P}\left(S_{\nu_n}\in dy,\tau_x>\nu_n,\nu_n\leq n^{1-\varepsilon}\right)\mathbf{P}(\tau_y>n-n^{1-\varepsilon}).
\end{align}
Applying now Lemmas~\ref{lem5}, we obtain 
\begin{align}\label{T1.3}
\nonumber
&\mathbf{P}(\tau_x>n;\nu_n\leq n^{1-\varepsilon})\\
\nonumber
&=
\frac{\varkappa+o(1)}{n^{k(k-1)/4}}\mathbf{E}\left[\Delta(x+S_{\nu_n});\tau_x>\nu_n,|S_{\nu_n}|\leq \theta_n\sqrt{n},\nu_n\leq n^{1-\varepsilon}\right]\\
\nonumber
&\hspace{0.5cm}+O\left(\frac{1}{n^{k(k-1)/4}}\mathbf{E}\left[\Delta(x+S_{\nu_n});\tau_x>\nu_n,|S_{\nu_n}|> \theta_n\sqrt{n},\nu_n\leq n^{1-\varepsilon}\right]\right)\\
\nonumber
&=
\frac{\varkappa+o(1)}{n^{k(k-1)/4}}\mathbf{E}\left[\Delta(x+S_{\nu_n});\tau_x>\nu_n,\nu_n\leq n^{1-\varepsilon}\right]\\
&\hspace{0.5cm}+O\left(\frac{1}{n^{k(k-1)/4}}\mathbf{E}\left[\Delta(x+S_{\nu_n});\tau_x>\nu_n,|S_{\nu_n}|> \theta_n\sqrt{n},\nu_n\leq n^{1-\varepsilon}\right]\right).
\end{align}
We now show that the first expectation converges to $V(x)$ and that the second expectation is negligibly small.
\begin{lemma}\label{lem11}
Under the assumptions of Theorem~\ref{T},
$$
\lim_{n\to\infty}\mathbf{E}\left[\Delta(x+S_{\nu_n}){\bf 1}\{\tau_x>\nu_n\};\nu_n\leq n^{1-\varepsilon}\right]=V(x).
$$
\end{lemma}
\begin{proof}
Rearranging, we have
\begin{align}\label{T1.4}
\nonumber
&\mathbf{E}\left[\Delta(x+S_{\nu_n}){\bf 1}\{\tau_x>\nu_n\};\nu_n\leq n^{1-\varepsilon}\right]\\
\nonumber
&\hspace{1cm}
=\mathbf{E}\left[\Delta(x+S_{\nu_n\wedge n^{1-\varepsilon}}){\bf 1}\{\tau_x>\nu_n\wedge n^{1-\varepsilon}\};\nu_n\leq n^{1-\varepsilon}\right]\\
&\hspace{1cm}
\nonumber
=\mathbf{E}\left[\Delta(x+S_{\nu_n\wedge n^{1-\varepsilon}}){\bf 1}\{\tau_x>\nu_n\wedge n^{1-\varepsilon}\}\right]\\
&\hspace{2cm}-\mathbf{E}\left[\Delta(x+S_{n^{1-\varepsilon}}){\bf 1}\{\tau_x>n^{1-\varepsilon}\};\nu_n> n^{1-\varepsilon}\right].
\end{align}
According to Lemma~\ref{lem2},
\begin{equation}
\label{T1.5}
\left|\mathbf{E}\left[\Delta(x+S_{n^{1-\varepsilon}}){\bf 1}\{\tau_x>n^{1-\varepsilon}\};\nu_n> n^{1-\varepsilon}\right]\right|
\leq C(x)\exp\{-C n^\varepsilon\}.
\end{equation}
Further,
\begin{align*}
&\mathbf{E}\left[\Delta(x+S_{\nu_n\wedge n^{1-\varepsilon}}){\bf 1}\{\tau_x>\nu_n\wedge n^{1-\varepsilon}\}\right]\\
&\hspace{1cm}
=\mathbf{E}\left[\Delta(x+S_{\nu_n\wedge n^{1-\varepsilon}})\right]
-\mathbf{E}\left[\Delta(x+S_{\nu_n\wedge n^{1-\varepsilon}}){\bf 1}\{\tau_x\leq\nu_n\wedge n^{1-\varepsilon}\}\right]\\
&\hspace{1cm}
=\Delta(x)
-\mathbf{E}\left[\Delta(x+S_{\nu_n\wedge n^{1-\varepsilon}}){\bf 1}\{\tau_x\leq\nu_n\wedge n^{1-\varepsilon}\}\right]\\
&\hspace{1cm}
=\Delta(x)
-\mathbf{E}\left[\Delta(x+S_{\tau_x}){\bf 1}\{\tau_x\leq\nu_n\wedge n^{1-\varepsilon}\}\right],
\end{align*}
here we have used the martingale property of $\Delta(x+S_{n})$.
Noting that $\nu_n\wedge n^{1-\varepsilon}\to\infty$ almost surely, we have
$$
\Delta(x+S_{\tau_x}){\bf 1}\{\tau_x\leq\nu_n\wedge n^{1-\varepsilon}\}\to\Delta(x+S_{\tau_x}).
$$
Then, using the integrability of $\Delta(x+S_{\tau_x})$ and the dominated convergence, we obtain
\begin{equation}
\label{T1.6}
\mathbf{E}\left[\Delta(x+S_{\tau_x}){\bf 1}\{\tau_x\leq\nu_n\wedge n^{1-\varepsilon}\}\right]
\to\mathbf{E}\left[\Delta(x+S_{\tau_x})\right].
\end{equation}
Combining (\ref{T1.4})--(\ref{T1.6}), we finish the proof of the lemma.
\end{proof}
\begin{lemma}
Under the assumptions of Theorem~\ref{T},
$$
\lim_{n\to\infty}\mathbf{E}\left[\Delta(x+S_{\nu_n});\tau_x>\nu_n,|S_{\nu_n}|> \theta_n\sqrt{n},\nu_n\leq n^{1-\varepsilon}\right]
=0.
$$
\end{lemma}
\begin{proof}
We first note that
\begin{align*}
&\mathbf{E}\left[\Delta(x+S_{\nu_n});\tau_x>\nu_n,|S_{\nu_n}|> \theta_n\sqrt{n},\nu_n\leq n^{1-\varepsilon}\right]\\
&\leq \mathbf{E}\left[\Delta(x+S_{\nu_n});T_x>\nu_n,|S_{\nu_n}|> \theta_n\sqrt{n},\nu_n\leq n^{1-\varepsilon}\right]\\
&\leq \mathbf{E}\left[\Delta(x+S_{n^{1-\varepsilon}});T_x>n^{1-\varepsilon},M_{n^{1-\varepsilon}}> \theta_n\sqrt{n}\right],
\end{align*}
where we used the submartingale property of $\Delta(x+S_{j}){\bf 1}\{T_x>j\}$, see Lemma~\ref{lem-1}.
(Recall that $M_j=\max_{i\leq j,r\leq k}|S_i^{(r)}|$.) Therefore, it is sufficient to show that
\begin{equation}\label{T1.7}
\mathbf{E}\left[\Delta(x+S_{n});T_x>n,M_{n}>n^{1/2+2\delta}\right]\to0
\end{equation}
for any positive $\delta$.

Define 
$$
A_n=\left\{\max_{1\le i\le n, 1\le j\le k}|\xi_i^{(j)}|\le n^{1/2+\delta}\right\}.
$$
Then
$$
\mathbf{E}\left[\Delta(x+S_{n});T_x>n,M_{n}>n^{1/2+2\delta},A_n\right]\leq
\mathbf{E}\left[|\Delta(x+S_{n})|;M_{n}>n^{1/2+2\delta},A_n\right].
$$
Since $|S_n^{(j)}|\leq n\max_{i\leq n}|\xi_i^{(j)}|\leq n^{3/2+\delta}$ on the event $A_n$, we arrive at the
following upper bound
\begin{align*}
&\mathbf{E}\left[\Delta(x+S_{n});T_x>n,M_{n}>n^{1/2+2\delta},A_n\right]\\
&\leq
C(x)\left(n^{3/2+\delta}\right)^{k(k-1)/2}\mathbf{P}(M_{n}>n^{1/2+2\delta},A_n).
\end{align*}
Applying now one of the Fuk-Nagaev inequalities, see Corollary 1.11 in \cite{Nag79}, we have
$$
\mathbf{P}(M_{n}>n^{1/2+2\delta},A_n)\leq\exp\{-Cn^\delta\}.
$$
As a result,
\begin{equation}\label{T1.8}
\lim_{n\to\infty}\mathbf{E}\left[\Delta(x+S_{n});T_x>n,M_{n}>n^{1/2+2\delta},A_n\right]=0
\end{equation}

Define
$$
\Sigma_l:=\sum_{i=1}^l\sum_{j=1}^k{\bf 1}\{|\xi_i^{(j)}|>n^{1/2+\delta}\},\quad l\leq n
$$
and 
$$
\Sigma_{l,n}:=\sum_{i=l+1}^n\sum_{j=1}^k{\bf 1}\{|\xi_i^{(j)}|>n^{1/2+\delta}\},\quad l<n.
$$
We note that
\begin{align}\label{T1.9}
\nonumber
\mathbf{E}\left[\Delta(x+S_{n});T_x>n,M_{n}>n^{1/2+2\delta},\overline{A_n}\right]\leq
\mathbf{E}\left[\Delta(x+S_{n})\Sigma_n;T_x>n\right]\\
=\mathbf{E}\left[\Delta(x+S_{n})\Sigma_n\right]-\mathbf{E}\left[\Delta(x+S_{n})\Sigma_n;T_x\leq n\right].
\end{align}
Since the conditioned distribution of $S_n$ given $\Sigma$ is exchangeable, we may apply Theorem 2.1
of \cite{KOR02}, which says that 
$$
\mathbf{E}[\Delta(x+S_l)|\Sigma_l]=\Delta(x),\quad l\leq n.
$$
Therefore,
\begin{equation}\label{T1.10}
\mathbf{E}[\Delta(x+S_l)]=\Delta(x)\mathbf{E}[\Sigma_l]
=k\Delta(x)l\mathbf{P}(|\xi|>n^{1/2+\delta}),\quad l\leq n.
\end{equation}
Using this equality and conditioning on $\mathcal{F}_l$, we have
\begin{align*}
&\mathbf{E}\left[\Delta(x+S_{n})\Sigma_n;T_x=l\right]=\mathbf{E}\left[\Delta(x+S_{n})\Sigma_l;T_x=l\right]+
\mathbf{E}\left[\Delta(x+S_{n})\Sigma_{l,n};T_x=l\right]\\
&=\mathbf{E}\left[\Delta(x+S_{l})\Sigma_l;T_x=l\right]+\mathbf{E}\left[\mathbf{E}[\Delta(x+S_{n})\Sigma_{l,n}|\mathcal{F}_l];T_x=l\right]\\
&=\mathbf{E}\left[\Delta(x+S_{l})\Sigma_l;T_x=l\right]+\mathbf{E}\left[\Delta(x+S_{l});T_x=l\right]\mathbf{E}\Sigma_{l,n},
\end{align*}
Consequently,
\begin{align*}
\mathbf{E}\left[\Delta(x+S_{n})\Sigma;T_x\leq n\right]&=\mathbf{E}\left[\Delta(x+S_{T})\Sigma_T;T_x\leq n\right]\\
&\hspace{1cm}+O\left(n\mathbf{P}(|\xi|>n^{1/2+\delta})\mathbf{E}\left[\Delta(x+S_{T});T_x\leq n\right]\right)\\
&=\mathbf{E}\left[\Delta(x+S_{T})\Sigma_T;T_x\leq n\right]+o(1).
\end{align*}
Finally,
$$
|\mathbf{E}\left[\Delta(x+S_{T})\Sigma_T;T_x\leq n\right]|\leq
\mathbf{E}\left[|\Delta(x+S_{T})|\Sigma_n\right]=o(1),
$$
by the dominated convergence, since $\Sigma_n\to0$. This implies that
\begin{equation}
\label{T1.11}
\mathbf{E}\left[\Delta(x+S_{n})\Sigma;T_x\leq n\right]=o(1).
\end{equation}
Combining (\ref{T1.9})--(\ref{T1.11}), we see that the left hand side of (\ref{T1.9}) converges to zero.
Then, taking into account (\ref{T1.8}), we get (\ref{T1.7}). Thus, the proof is finished.
\end{proof}
Now we are in position to complete the proof of Theorem~\ref{T}.
It follows from the lemmas and (\ref{T1.1}) and (\ref{T1.3}) that
$$
\mathbf{P}(\tau_x>n)=\frac{\varkappa V(x)}{n^{k(k-1)/4}}(1+o(1)).
$$
%%%%%%%%%%%%%%%%%%%%%%%%%%%%%%%%%%%%%%%%%%%%%%%%%%%%%%%%%%%%%%%%%%%%%%%%%%%%%%%%%%%%%%%%%%%%%%%%%%%%%%%%%%%%%%%%%%%%
%%%%%%%%%%%%%%%%%%%%%%%%%%%%%%%%%%%%%%%%%%%%%%%%%%%%%%%%%%%%%%%%%%%%%%%%%%%%%%%%%%%%%%%%%%%%%%%%%%%%%%%%%%%%%%%%%%%
\section{Weak convergence results}\label{sect.weak.convergence}
\begin{lemma}
For any $x\in W$, the distribution 
$\mathbf P\left(\frac{x+S_n}{\sqrt n} \in \cdot | \tau_x>n\right)$ weakly converges to the 
distribution with the density $\frac{1}{Z_1}e^{-|y|^2/2}\Delta(y)$, where $Z_1$ is the norming constant.
\end{lemma}
\begin{proof}
We need to show that 
\begin{eqnarray}\label{eq.mmm}
\frac{\mathbf P(x+S_n\in \sqrt n A,\tau_x>n)}{\mathbf P(\tau_x>n)}\to 
Z_1^{-1}\int_A e^{-|y|^2/2}\Delta(y)dy.
\end{eqnarray}
First note that, as in (\ref{T1.1}) and 
(\ref{T1.3}),
\begin{align*}
&\mathbf P(x+S_n\in \sqrt n A,\tau_x>n)
=\mathbf{P}(\tau_x>n,x+S_n\in \sqrt n A,\nu_n\leq n^{1-\varepsilon})+O\left(e^{-Cn^\varepsilon}\right)\\
&\hspace{1cm}=\mathbf{P}(\tau_x>n,x+S_n\in \sqrt n A, |S_{\nu_n}|\le \theta_n\sqrt n,\nu_n\leq n^{1-\varepsilon})+o(\mathbf P(\tau_x>n)).
\end{align*}
Next,
\begin{align*}
 &
\mathbf{P}(\tau_x>n,x+S_n\in \sqrt n A, |S_{\nu_n}|\le \theta_n\sqrt n,\nu_n\leq n^{1-\varepsilon})\\
&=\sum_{j=1}^{n^{1-\varepsilon}}\int_{W_{n,\varepsilon}\cap \{|y|\le \theta_n \sqrt n\}}\mathbf{P}(\tau_x>k,x+S_k\in y\in \sqrt n A,\nu_n=k)\\
 &\hspace{4cm} \times\mathbf{P}(\tau_y>n-k,y+S_{n-k}\in \sqrt n A).
\end{align*}
Using the coupling and arguing  as in Lemma~\ref{lem5}, 
one can show that 
$$
\mathbf{P}(\tau_y>n-k,y+S_{n-k}\in \sqrt n A)\sim
\mathbf{P}(\tau^{bm}_y>n,y+B_{n}\in \sqrt n A)
$$
uniformly in $k\le n^{1-\varepsilon}$ and $y\in W_{n,\varepsilon}$.
Next we apply asymptotics (\ref{L6.3}) and obtain that 
$$
\mathbf{P}(\tau_y>n-k,y+S_{n-k}\in \sqrt n A)
\sim K\int_A dz e^{-|z|^2/2}\Delta(y)\Delta(z) n^{-k(k-1)/4}
$$
uniformly in $y\in W_{n,\varepsilon},\ |y|\le \theta_n\sqrt n$.
As a result we obtain
\begin{align*}
& \mathbf P(x+S_n\in \sqrt n A,\tau_x>n)\sim
\int_A dz e^{-|z|^2/2}\Delta(z) n^{-k(k-1)/4}\\
&\hspace{1cm}\times K\mathbf{E}[\Delta(S_{\nu_n})\tau_x>n,x+S_n\in \sqrt n A, |S_{\nu_n}|\le \theta_n\sqrt n,\nu_n\leq n^{1-\varepsilon}]\\
&\hspace{1cm}\sim K\int_A dz e^{-|z|^2/2}\Delta(z) n^{-k(k-1)/4} V(x) ,
 \end{align*}
where the latter equivalence holds due to Lemma~\ref{lem11}.
Substituting the latter equivalence in (\ref{eq.mmm}) and using the asymptotics for $\mathbf P(\tau_x>n)$, we arrive at the conclusion.
\end{proof}

Now we change slightly notation. Let 
$$
\mathbf P_x(S_n\in A)=\mathbf P(x+S_n\in A).
$$
 
\begin{lemma}\label{lem.weak.convergence}
Let $X^n(t)=\frac{S_{[nt]}}{\sqrt n}$ be the family of processes  
with the probability measure $\mathbf{\widehat P}^{(V)}_{x\sqrt n},x \in W$.
Then $X^n$ weakly converges in $C(0,\infty)$ to the Dyson Brownian motion with starting point $x$, 
i.e. to the process distributed according to the probability measure $\mathbf {\widehat P}_x^{(\Delta)}$.
\end{lemma}
\begin{proof}
The proof is given via coupling from Lemma~\ref{lem4}.
To prove the claim we need to show that the convergence take place in 
$C[0,l]$ for every l. 
The proof is identical for $l$, so we let $l=1$ to simplify notation.
Thus it sufficient to show that  for every function $f:0\le f\le 1$ uniformly 
continuous on $C[0,1]$,
$$
\mathbf {\widehat E}_{x\sqrt n}^{(V)} f(X^n)\to 
\mathbf {\widehat E}_{x}^{(\Delta)} f(B)\quad\text{as } n\to\infty.
$$
By Lemma~\ref{lem4} one can define $B_n$ and $S_n$ on the same probability 
in such a way that the complement of the event
$$
A_n=\{\sup_{u\le n}|S_{[u]}-B_{u}|\le n^{1/2-a}\} 
$$
is negligible: 
$$
\mathbf P(\overline A_n)=o(n^{-\gamma})
$$
for some $a>0$ and $\gamma>0$.
Let $B^n_t=B_{nt}/\sqrt n$. By the scaling property of the Brownian motion 
$\mathbf {\widehat E}_{x}^{(\Delta)} f(B)
=\mathbf {\widehat E}_{x\sqrt n}^{(\Delta)} f(B^n)$.

Split the expectation into two parts,
\begin{align*}
\mathbf {\widehat E}_{x\sqrt n}^{(V)} f(X^n)
=\mathbf {\widehat E}_{x\sqrt n}^{(V)} [f(X^n);A_n]
+
\mathbf {\widehat E}_{x\sqrt n}^{(V)} [f(X^n);\overline A_n]\equiv E_1+E_2.
\end{align*}
Since the function $f$ is uniformly continuous, 
$$
|f(X^n)-f(B^n)|\le C\sup_{0\le u\le 1} |X^n_u-B^n_u|\le Cn^{-a}
$$
on the event $A_n$. Then,
\begin{align*}
 &\frac{1}{V(x\sqrt n)}\mathbf { E}_{x\sqrt n} [(f(X^n)-f(B^n))V(S_n);\tau>n,A_n]\\
&\le Cn^{-a} \frac{\mathbf { E}_{x\sqrt n} [V(S_n);\tau>n,A_n]}{V(x\sqrt n)}
\le Cn^{-a} \frac{\mathbf { E}_{x\sqrt n} [V(S_n);\tau>n]}{V(x\sqrt n)}
=Cn^{-a}
\end{align*}
tends to $0$ as $n\to \infty$. Therefore,
\begin{align*}
 E_1&=\frac{1}{V(x\sqrt n)}\mathbf { E}_{x\sqrt n} [f(X^n)V(S_n);\tau>n,A_n]\\
&=o(1)+\frac{1}{V(x\sqrt n)}\mathbf { E}_{x\sqrt n} [f(B^n)V(S_n);\tau>n,A_n].
\end{align*}
Moreover, on the event $A_n$ hold the following inequalities 
$$
B^{(j)}_i-B^{(j-1)}_i-2n^{1/2-a}\le S^{(j)}_i-S^{(j-1)}_i\le B^{(j)}_i-B^{(j-1)}_i+2n^{1/2-a}
$$
for $1\le i\le n$ and $2\le j\le k$. Let 
$x^{\pm}_n=(x\sqrt n\pm 2(j-1)n^{1/2-a})$.
Arguing as in Lemma~\ref{lem5} and using monotonicity of $V$, we obtain
\begin{align*}
& \frac{1}{V(x\sqrt n)}\mathbf { E}_{x\sqrt n} [f(B^n)V(S_n);\tau>n,A_n]\\
&\hspace*{1cm}\le 
\frac{(1+o(1))}{V(x \sqrt n)}\mathbf { E}_{x^+_n} [f(B^n)V(B_n);\tau^{bm}>n,A_n]
\\
&\hspace*{1cm}\le 
\frac{(1+o(1))}{V(x \sqrt n)}\mathbf { E}_{x^+_n} [f(B^n)\Delta(B_n);\tau^{bm}>n,A_n]\\
&\hspace*{1cm}=(1+o(1))
\frac{V(x^+_n)}{V(x \sqrt n)}\mathbf {\widehat E}_{x^+_n}^{\Delta} [f(B^n);A_n]
=(1+o(1))
\mathbf {\widehat E}_{x^+_n}^{\Delta} [f(B^n);A_n],
\end{align*}
where we used (d) of Proposition~\ref{prop1} in the second and the third lines. 
Replacing $x^+$ with $x^-$, one can easily obtain the following lower bound
$$
\frac{1}{V(x\sqrt n)}\mathbf { E}_{x\sqrt n} [f(B^n)V(S_n);\tau>n,A_n]
\ge (1+o(1))\mathbf {\widehat E}_{x^-_n}^{\Delta} [f(B^n);A_n].
$$
Note also that 
\begin{align*}
\mathbf {\widehat E}_{x^\pm_n}^{\Delta} [f(B^n);A_n]
&=
\mathbf {\widehat E}_{x^\pm_n}^{\Delta} [f(B^n)]
-\mathbf {\widehat E}_{x^\pm_n}^{\Delta} [f(B^n);\overline A_n]
\\
&=
(1+o(1))
\mathbf {\widehat E}_{x\sqrt n}^{\Delta} [f(B^n)]
-\mathbf {\widehat E}_{x^\pm_n}^{\Delta} [f(B^n);\overline A_n]
\end{align*}
Therefore,
$$
|E_1-
\mathbf {\widehat E}_{x\sqrt n}^{\Delta} [f(B^n)]|\le 
o(1)+\mathbf {\widehat E}_{x^+_n}^{\Delta} [f(B^n);\overline A_n]+
\mathbf {\widehat E}_{x^-_n}^{\Delta} [f(B^n);\overline A_n].
$$
Thus, if we show that 
$$
\mathbf {\widehat E}_{x^\pm_n}^{\Delta} [f(B^n);\overline A_n]=o(1),\quad\text{and}\quad 
E_2=o(1),
$$
we are done. Since the proofs of these statements are almost identical we concentrate 
on showing that $E_2=o(1).$
We have, since $f\le 1$, 
\begin{align*}
 E_2\le 
\frac{1}{V(x\sqrt n)}\mathbf { E}_{x\sqrt n} [V(S_n);|S_n|\le n^{1/2+\delta},\overline{A_n}]\\
+\frac{1}{V(x\sqrt n)}\mathbf { E}_{x\sqrt n} [V(S_n);|S_n|> n^{1/2+\delta}].
\end{align*}
Put $y_n=(2n^{1/2+\delta},\ldots, 2(k-1)n^{1/2+\delta})$.
Then,
\begin{align}\label{eq153}
\nonumber 
 &\frac{1}{V(x\sqrt n)}\mathbf { E}_{x\sqrt n} [V(S_n);|S_n|\le n^{1/2+\delta},\overline{A_n}]
\le 
\frac{V(x_n+y_n)}{V(x\sqrt n)}\mathbf { P}_{x\sqrt n} (|S_n|\le n^{1/2+\delta},\overline{A_n})
\\
&\hspace{2cm}\le 
C \frac{\Delta_1(x_n+y_n)}{\Delta(x\sqrt n)}\mathbf { P}_{x\sqrt n} 
(\overline{A_n})\le C n^{\delta k(k-1)/2}n^{-\gamma}\to 0,
\end{align}
if we pick $\delta $ sufficiently small.
Next, using the bounds $V(x)\leq V^{(T)}(x)\leq\Delta_1(x)$, we get
\begin{align*}
 \mathbf { E}_{x\sqrt n} [V(S_n);|S_n|> n^{1/2+\delta}]
\le \sum_{j=1}^k \mathbf { E} [\Delta_1(x\sqrt n+S_n);|S_n^{(j)}|> n^{1/2+\delta}/k].
\end{align*}
Arguing similarly to the second part of Lemma~\ref{lem0}, one can see that 
 \begin{align*}
 &\mathbf { E} [\Delta_1(x\sqrt n+S_n);|S_n^{(j)}|> n^{1/2+\delta}/k]\\
&\hspace{2cm}\le C(x)\sum_{\mathcal J\subset\mathcal P} n^{|\mathcal J|/2}
\prod_{\mathcal P \backslash \mathcal J} \mathbf 
\mathbf { E} [|S_n^{(j_2)}-S_n^{(j_1)}|;|S_n^{(j)}|> n^{1/2+\delta}/k].
\end{align*}
The expectation of the product can be estimated exactly as in Lemma~\ref{lem0} using 
the Fuk-Nagaev inequality. 
This gives us 
$$
 \frac{1}{V(x\sqrt n)}\mathbf { E}_{x\sqrt n} [V(S_n);|S_n|> n^{1/2+\delta}]
=\frac{o(n^\frac{k(k-1)}{4})}{\Delta(x\sqrt n)}=o(1).
$$
Thus, the proof is finished.
\end{proof}

Now we consider start from a fixed point $x$. 
\begin{lemma}\label{lem.weak.convergence2}
Let $X^n(t)=\frac{S_{[nt]}}{\sqrt n}$ be the family of processes  
with the probability measure $\mathbf{\widehat P}^{(V)}_{x},x \in W$.
Then $X^n$ converges weakly to the Dyson Brownian motion with starting point $0$. 
\end{lemma}
\begin{proof}
As in the proof of the previous lemma, wee show the convergence on $C[0,1]$ only.
It sufficient to show that  for every function $f:0\le f\le 1$ uniformly 
continuous on $C[0,1]$,
$$
\mathbf {\widehat E}_{x}^{(V)} f(X^n)\to 
\mathbf {\widehat E}_{0} f(B)\quad\text{as } n\to\infty.
$$

First,
\begin{align*}
\mathbf {\widehat E}_{x}^{(V)} [f(X^n)]
=
\mathbf {\widehat E}_{x}^{(V)}[ f(X^n),\nu_n\le n^{1-\varepsilon}]
+
\mathbf {\widehat E}_{x}^{(V)}[ f(X^n),\nu_n>n^{1-\varepsilon}].
\end{align*}
The second term 
\begin{align*}
\mathbf {\widehat E}_{x}^{(V)}[ f(X^n),\nu_n>n^{1-\varepsilon}]
&\le 
\mathbf {\widehat P}_{x}^{(V)}(\nu_n>n^{1-\varepsilon})
=\frac{\mathbf E[ V(x+S_n);\tau_x>\nu_n,\nu_n>n^{1-\varepsilon}]}{V(x)}
\\
&\le 
C \frac{\mathbf E[ \Delta_1(x+S_n);\tau_x>\nu_n,\nu_n>n^{1-\varepsilon}]}{V(x)}\to 0,
\end{align*}
where the latter convergence follows from Lemma~\ref{lem2}. Next,
\begin{align*}
\mathbf {\widehat E}_{x}^{(V)}[ f(X^n);\nu_n\le n^{1-\varepsilon}]
&=
\mathbf {\widehat E}_{x}^{(V)}[ f(X^n);\nu_n\le n^{1-\varepsilon}, M_{\nu_n}\le \theta_n\sqrt n]
\\
&\hspace{2cm}+
\mathbf {\widehat E}_{x}^{(V)}[ f(X^n);\nu_n\le n^{1-\varepsilon}, M_{\nu_n}> \theta_n\sqrt n].
\end{align*}
Then,
\begin{align*}
 &\mathbf {\widehat E}_{x}^{(V)}[ f(X^n);\nu_n\le n^{1-\varepsilon}, M_{\nu_n}> \theta_n\sqrt n]
\le 
\mathbf {\widehat P}_{x}^{(V)}(\nu_n\le n^{1-\varepsilon}, M_{\nu_n}> \theta_n\sqrt n)
\\
&\hspace{2cm}=
\frac{\mathbf {E}(V(S_{\nu_n});\nu_n\le n^{1-\varepsilon}, M_{\nu_n}> \theta_n\sqrt n)}{V(x)}
\\
&\hspace{2cm}\le (1+o(1)) 
\frac{\mathbf {E}(\Delta(S_{\nu_n});\nu_n\le n^{1-\varepsilon}, M_{\nu_n}> \theta_n\sqrt n)}{V(x)}\to 0,
\end{align*}
by (\ref{T1.7}). These preliminary estimates give us 
\begin{align}\label{eq.preliminary}
 \mathbf {\widehat E}_{x}^{(V)} [f(X^n)]
=
\mathbf {\widehat E}_{x}^{(V)}[ f(X^n);\nu_n\le n^{1-\varepsilon}, M_{\nu_n}\le \theta_n\sqrt n]+o(1).
\end{align}
Next let 
$$
f(y,k, X^n)=f\left(\frac{y}{\sqrt n}\mathbf 1_{\{t\leq k/n\}}+X^n(t)\mathbf 1_{\{t> k/n\}}\right).
$$
It is not difficult to see that 
on the event 
$\{x+S_{\nu_n}\in dy,M_{\nu_n}\le \theta_n\sqrt n\}$, the following  holds 
$$
f(y,k, X^n)-f(X^n)=o(1)
$$ 
uniformly in $|y|\leq\theta_n\sqrt n$ and $k\leq n^{1-\varepsilon}$. 
Therefore,
\begin{align*}
&\mathbf {\widehat E}_{x}^{(V)}[ f(X^n);\nu_n\le n^{1-\varepsilon}, M_{\nu_n}\le \theta_n\sqrt n]\\
&\sim
\mathbf {\widehat E}_{x}^{(V)}[ f(S_{\nu_n},\nu_n,X^n);\nu_n\le n^{1-\varepsilon}, M_{\nu_n}\le \theta_n\sqrt n]
\\
&=\sum_{k\leq n^{1-\varepsilon}}\int_{W_{n,\varepsilon}}
\mathbf P\left(x+S_{k}\in dy,\tau_x>k,\nu_n=k, M_{\nu_n}\le \theta_n\sqrt n\right)
\frac{V(y)}{V(x)}\\
&\hspace{2cm}\times\mathbf {\widehat E}_{y}^{(V)}f\left(\frac{y}{\sqrt n}\mathbf 1_{\{t\leq k/n\}}+X^n(t-k/n)\mathbf 1_{\{t> k/n\}}\right). 
\end{align*}
Using coupling arguments from Lemma \ref{lem.weak.convergence}, one can easily get
\begin{align*}
&\mathbf{\widehat E}_{y}^{(V)}f\left(\frac{y}{\sqrt n}\mathbf 1_{\{t\leq k/n\}}+X^n(t-k/n)\mathbf 1_{\{t> k/n\}}\right)\\
&\hspace{2cm}
\sim\mathbf{\widehat E}_{y}^{(\Delta)}f\left(\frac{y}{\sqrt n}\mathbf 1_{\{t\leq k/n\}}+B^n(t-k/n)\mathbf 1_{\{t> k/n\}}\right).
\end{align*}
Using results of Section 4 of \cite{OY02}, one has
$$
\mathbf{\widehat E}_{y}^{(\Delta)}f\left(\frac{y}{\sqrt n}\mathbf 1_{\{t\leq k/n\}}+B^n(t-k/n)\mathbf 1_{\{t> k/n\}}\right)
\sim \mathbf{\widehat E}_{0}^{(\Delta)}[f(B)].
$$
Consequently,
\begin{align*}
&\mathbf {\widehat E}_{x}^{(V)}[ f(X^n);\nu_n\le n^{1-\varepsilon}, M_{\nu_n}\le \theta_n\sqrt n]\\
&\sim\mathbf{\widehat E}_{0}^{(\Delta)}[f(B)]
\frac{\mathbf {E}[V(x+S_{\nu_n});\tau_x>\nu_n, \nu_n\le n^{1-\varepsilon}, M_{\nu_n}\le \theta_n\sqrt n]}{V(x)}\\
&\sim\mathbf{\widehat E}_{0}^{(\Delta)}[f(B)]
\frac{\mathbf {E}[\Delta(x+S_{\nu_n});\tau_x>\nu_n, \nu_n\le n^{1-\varepsilon}, M_{\nu_n}\le \theta_n\sqrt n]}{V(x)}.
\end{align*}
Using now Lemma \ref{lem11} and relation (\ref{T1.7}), we get finally
$$
\mathbf {\widehat E}_{x}^{(V)}[ f(X^n);\nu_n\le n^{1-\varepsilon}, M_{\nu_n}\le \theta_n\sqrt n]
\sim\mathbf{\widehat E}_{0}^{(\Delta)}[f(B)].
$$
Combining this with (\ref{eq.preliminary}), we complete the proof of the lemma.
\end{proof}

{\bf Acknowledgement} The research of Denis Denisov was supported by EPSRC grant No.~EP/E033717/1. 
This work was carried out during the visits of the first author to Technische Universit{\"a}t M{\"u}nchen and Ludwig-Maximilians-Universit{\"a}t M{\"u}nchen and the visit of the second author to Heriot-Watt University. Both authors are grateful to the above-mentioned institutions for their hospitality.

%%%%%%%%%%%%%%%%%%%%%%%%%%%%%%%%%%%%%%%%%%%%%%%%%%%%%%%%%%%%%%%%%%%%%%%%%%%%%%%%%%%%%%%%%%%%%%%%%%%%%%%%%%%%%%%%%%%%


\begin{thebibliography}{99}
\bibitem{BS06} Baik, J. and Suidan, T.
\newblock Random matrix central limit theorems for non-intersetcting random walks.
\newblock \emph{Ann. Probab.}, 35:1807--1834, 2007.

\bibitem{BD94} Bertoin, J. and Doney, R.A.
\newblock On conditioning a random walk to stay nonnegative.
\newblock \emph{Ann. Probab.}, 22:2152--2167, 1994.

\bibitem{BM05} Bodineau, T. and Martin, J.
\newblock A universality property for last-passage percolation paths close to the axis .
\newblock {\em Electron. Comm. in Probab.}, 10:105--112, 2005.


\bibitem{Bor83} Borisov, I.S.
\newblock On the question of the rate of convergence in the Donsker--Prokhorov invariance principle.
\newblock \emph{Theory Probab. Appl.}, 28:388--392, 1983.

\bibitem{Bor72} Borovkov, A.A.
\newblock Notes on inequalities for sums of independent random variables.
\newblock \emph{Theory Probab. Appl.}, 17:556--557, 1972.

\bibitem{BD06} Bryn-Jones, A. and Doney, R.A.
\newblock A functional limit theorem for random walk conditioned to stay non-negative.
\newblock\emph{J. London Math. Soc. (2)}, 74:244--258, 2006.

\bibitem{Dy62} Dyson F.J.
\newblock A Brownian-motion model for the eigenvalues of a random matrix.
\newblock \emph{J. Math. Phys}, 3:1191--1198, 1962.

\bibitem{Grab99} Grabiner, D.J.
\newblock Brownian motion in a Weyl chamber, non-colliding particles, and random matrices.
\newblock{\em Ann. Inst. H. Poincare Probab. Statist.}, 35(2):177--204, 1999.

\bibitem{EK08} Eichelsbacher, P. and K{\"o}nig, W.
\newblock Ordered random walks.
\newblock{\em Electron. J. Probab.}, 13:1307--1336, 2008.

\bibitem{M76} Major, P.
\newblock The approximation of partial sums of rv's.
\newblock\emph{Z. Wahrscheinlichkeitstheorie verw. Gebiete}, 35:213--220,1976.

\bibitem{KOR02} K{\"o}nig, W., O'Connell, N. and Roch, S.
\newblock Non-colliding random walks, tandem queues, and discrete orthogonal polynomial ensembles.
\newblock {\em Electron. J. Probab.}, 7:1--24, 2002.

\bibitem{K05} K{\"o}nig, W.
\newblock Orthogonal polynomial ensembles in probability theory.
\newblock {\em Probab. Surv.}, 2:385--447, 2005.

\bibitem{Nag79} Nagaev, S.V.
\newblock Large deviations of sums of independent random variables.
\newblock{\em Ann. Probab.}, 7:745--789, 1979.

\bibitem{OY02} O'Connell, N. and Yor, M.
\newblock A representation for non-colliding random walks.
\newblock {\em Elect. Comm. Probab.}, 7:1--12, 2002.

\bibitem{PR08} Puchala, Z, Rolski, T.
\newblock The exact asymptotic of the collision time tail distribution for independent Brownian particles with different drifts.
\newblock{\em Probab. Theory Related Fields}, 142:595--617, 2008.

\bibitem{Sch09} Schapira, B.
\newblock Random walk on a building of type $\tilde{A}_r$ and Brownian motion on a Weyl chamber.
\newblock {\em Ann. Inst. H. Poincare Probab. Statist.}, 45:289-301, 2009.

\bibitem{Var99} Varopoulos, N.Th.
\newblock Potential theory in conical domains.
\newblock\emph{Math. Proc. Camb. Phil. Soc.}, 125:335--384, 1999.
\end{thebibliography}
\end{document}